\documentclass[a4paper,11pt]{amsart}
\usepackage[all]{xy}
\usepackage{amssymb, amscd}
\usepackage{amsopn}
\usepackage{amsmath}

\newcommand{\ie}{{\em i.e.}\ }
\newcommand{\cf}{{\em cf.}\ }

\newcommand{\ko}{\: , \;}

\newcommand{\ol}[1]{\overline{#1}}

\numberwithin{equation}{subsection}
\newtheorem{theorem}[subsection]{Theorem}
\newtheorem{lemma}[subsection]{Lemma}
\newtheorem{proposition}[subsection]{Proposition}
\newtheorem{corollary}[subsection]{Corollary}
\newtheorem{remark}[subsection]{Remark}
\newtheorem{example}[subsection]{Example}
\newcommand{\reminder}[1]{}

\newcommand{\opname}[1]{\operatorname{\mathsf{#1}}}

\renewcommand{\mod}{\opname{mod}\nolimits}

\newcommand{\proj}{\opname{proj}\nolimits}
\newcommand{\Mod}{\opname{Mod}\nolimits}

\newcommand{\per}{\opname{per}\nolimits}

\newcommand{\add}{\opname{add}\nolimits}
\newcommand{\op}{^{op}}

\newcommand{\im}{\opname{im}\nolimits}

\newcommand{\Z}{\mathbb{Z}}

\newcommand{\Q}{\mathbb{Q}}

\newcommand{\iso}{\stackrel{_\sim}{\rightarrow}}
\newcommand{\liso}{\stackrel{\,\,\,_\sim}{\leftarrow}}
\newcommand{\id}{\mathbf{1}}

\newcommand{\Hom}{\opname{Hom}}

\newcommand{\RHom}{\opname{RHom}}
\newcommand{\bp}{{\mathbf p}}

\newcommand{\Ext}{\opname{Ext}}

\newcommand{\End}{\opname{End}}

\newcommand{\ten}{\otimes}
\newcommand{\lten}{\overset{L}{\ten}}

\newcommand{\Tor}{\opname{Tor}}

\newcommand{\ca}{{\mathcal A}}
\newcommand{\cb}{{\mathcal B}}
\newcommand{\cc}{{\mathcal C}}
\newcommand{\cd}{{\mathcal D}}

\newcommand{\cF}{{\mathcal F}}

\newcommand{\ch}{{\mathcal H}}

\newcommand{\cm}{{\mathcal M}}
\newcommand{\cn}{{\mathcal N}}

\newcommand{\cp}{{\mathcal P}}
\newcommand{\cR}{{\mathcal R}}

\newcommand{\ct}{{\mathcal T}}

\newcommand{\cz}{{\mathcal Z}}

\newcommand{\eps}{\varepsilon}
\renewcommand{\phi}{\varphi}
\newcommand{\del}{\partial}

\renewcommand{\tilde}[1]{\widetilde{#1}}

\def\r{\rightarrow}

\def\ldb{\mathopen{\{\!\!\{}} \def\rdb{\mathclose{\}\!\!\}}}

\let\blb\mathbb

\def \TT{{\blb T}}

\def\id{\text{id}}

\def\Der{\operatorname{Der}}

\def\Mod{\operatorname{Mod}}
\def\mod{\operatorname{mod}}

\def\Ext{\operatorname {Ext}}
\def\Hom{\operatorname {Hom}}

\def\End{\operatorname {End}}
\def\RHom{\operatorname {RHom}}

\def\im{\operatorname {im}}

\def\Tor{\operatorname {Tor}}
\def\End{\operatorname {End}}

\def\id{{\operatorname {id}}}

\def\add{\operatorname {add}}

\def\r{\rightarrow}

\begin{document}

\date{August 25, 2009}

\title{Deformed Calabi-Yau completions}
\author{Bernhard Keller \\ \\
with an appendix by Michel Van den Bergh}
\address{Bernhard Keller\\
Universit\'e Paris Diderot -- Paris 7\\
UFR de Math\'ematiques\\
Institut de Math\'ematiques de Jussieu, UMR 7586 du CNRS \\
Case 7012\\
B\^{a}timent Chevaleret\\
75205 Paris Cedex 13\\
France } \email{keller@math.jussieu.fr}

\address{Michel Van den Bergh, Departement WNI, Limburgs Universitair Centrum,
3590 Diepenbeek, Belgium.}
\email{vdbergh@luc.ac.be}

\begin{abstract} We define and investigate deformed $n$-Calabi-Yau
  completions of homologically smooth differential graded (=dg)
  categories. Important examples are: deformed preprojective algebras
  of connected non Dynkin quivers, Ginzburg dg algebras associated to
  quivers with potentials and dg categories associated to the category
  of coherent sheaves on the canonical bundle of a smooth variety. We
  show that deformed Calabi-Yau completions do have the Calabi-Yau
  property and that their construction is compatible with derived
  equivalences and with localizations. In particular, Ginzburg dg
  algebras have the Calabi-Yau property. We show that deformed
  $3$-Calabi-Yau completions of algebras of global dimension at most
  $2$ are quasi-isomorphic to Ginzburg dg algebras and apply this to
  the study of cluster-tilted algebras and to the construction of
  derived equivalences associated to mutations of quivers with
  potentials. In the appendix, Michel Van den Bergh uses non
  commutative differential geometry to give an alternative proof of
  the fact that Ginzburg dg algebras have the Calabi-Yau property.
\end{abstract}


\maketitle

\tableofcontents

\section{Introduction}

\subsection{Context and main results} This article is motivated by the theory which
links cluster algebras \cite{FominZelevinsky02} to representations
of quivers and finite-dimensional algebras, \cf~\cite{Keller08c}
for a survey. In this theory, Calabi-Yau algebras and categories
play an important r\^ole.
For example,  Geiss-Leclerc-Schr\"oer
use the $2$-Calabi-Yau property of the category of modules over a
preprojective algebra (\cf \cite{GeissLeclercSchroeer08a}), Iyama-Reiten
\cite{IyamaReiten08} study mutations using tilting modules over
$2$- and $3$-Calabi-Yau algebras related to singularities
\cite{VandenBergh04} and Amiot's construction \cite{Amiot08a} of generalized
cluster categories relies on dg algebras which are $3$-Calabi-Yau
as bimodules.
The Calabi-Yau property is also important in Kontsevich-Soibelman's
recent interpretation of cluster transformations in their study of
Donaldson-Thomas invariants and stability structures \cite{KontsevichSoibelman08}.

Let us recall the definition of the Calabi-Yau property for algebras and
for triangulated categories:
Let $A$ be an (associative, unital) algebra over a field $k$. We identify $A$-bimodules
with (right) modules over the enveloping algebra $A^e=A\ten A\op$. Let $n$ be
an integer. Recall that the algebra $A$
is {\em homologically smooth} if, as a bimodule, it admits a finite resolution
by finitely generated projective bimodules. Following Ginzburg and Kontsevich
(Definition~3.2.3 of \cite{Ginzburg06}),
it is {\em $n$-Calabi-Yau as a bimodule} if it is homologically smooth
and, in the derived category of $A$-bimodules, we have an isomorphism
\[
f: A^\vee \iso A \mbox{ such that } f^\vee=f \ko
\]
where, for a bimodule complex $M$, we denote by $M^\vee$ the
derived bimodule dual shifted by $n$ degrees
\[
M^\vee = \Sigma^n \RHom_{A^e}(M,A^e).
\]
The bimodule complex
$\RHom_{A^e}(M,A^e)$ is the {\em inverse dualizing complex} of
\cite{VandenBergh97}. If $A$ is $n$-Calabi-Yau as a bimodule,
the subcategory $\cd_{fd}(A)$ of the derived category $\cd(A)$ formed
by the modules whose homology is of finite total dimension is
{\em $n$-Calabi-Yau as a triangulated category}, \ie we have
non degenerate bifunctorial pairings
\[
\langle\,, \rangle: \Hom(M,\Sigma^n L) \times \Hom(L, M) \to k
\]
such that, for $p+q=n$, we have
\[
\langle \Sigma^p f, g \rangle = (-1)^{pq} \langle \Sigma^q g, f \rangle
\]
for all $f: M \to \Sigma^q L$ and $g: L \to \Sigma^p M$, \cf \cite{Keller08d}.

Let $A$ be any homologically smooth algebra (or more generally: dg category),
and let $n$ be an integer. One of the main objects of study of this
paper is a canonical dg algebra $\Pi_n(A)$
which we call the {\em $n$-Calabi-Yau completion}
or the {\em derived $n$-preprojective algebra}.
If $\theta$ denotes a projective resolution of the shifted bimodule dual $A^\vee$, we simply put
$\Pi_n(A)$ equal to the tensor dg algebra
\[
\Pi_n(A) = T_A(\theta)= A \oplus \theta \oplus (\theta \ten_A \theta) \oplus \cdots \ .
\]
Under Koszul duality, this construction corresponds to Ed~Segal's cyclic
completion \cite{Segal08}.
If $A$ is the path algebra of a connected non Dynkin quiver and
$n=2$, one can show that $\Pi_n(A)$ is quasi-isomorphic to the preprojective algebra of $A$,
\cf section~4.2 of \cite{Keller08d}.
If $A$ is the endomorphism algebra of a tilting object in the derived category
of quasi-coherent sheaves on a smooth algebraic variety $X$ of dimension $n-1$
(or more generally, the derived endomorphism algebra of
any compact generator \cite{BondalVandenBergh03}), then
the derived category of $\Pi_n(A)$ is triangle equivalent to the derived
category of quasi-coherent sheaves on the total space of the canonical bundle of
$X$, \cf \cite{Takahashi08}. We will show that
$\Pi_n(A)$ is always $n$-Calabi-Yau as a bimodule and that the
construction $A \mapsto \Pi_n(A)$ is equivariant under derived Morita
equivalences and compatible with (dg) localizations.

Let $c$ be a Hochschild cycle of degree $n-2$ of $A$. It yields a
canonically defined morphism $\delta: \theta \to A$ of degree $1$.
We define the {\em deformed $n$-Calabi-Yau completion} or {\em
deformed derived $n$-preprojective algebra} $\Pi_n(A,c)$ to be
obtained from $\Pi_n(A)$ by deforming the differential of the tensor
algebra using $\delta$. More intrinsically, the dg algebra $\Pi_n(A,c)$
can be constructed as a homotopy pushout from the Calabi-Yau
completion $\Pi_{n-1}(A)$ as suggested in \cite{Davison09}.
One can show that deformed preprojective
algebras \cite{CrawleyBoeveyHolland98} of connected non Dynkin
quivers are obtained in this way for $n=2$. For $n=3$, the (non
complete) Ginzburg dg algebra (\cf section~4.2 of \cite{Ginzburg06})
associated with a quiver $Q$ and a potential $W$ becomes an example.
Indeed, it is quasi-isomorphic to $\Pi_3(kQ, c)$, where $c$ is the
image of $W$, considered as an element of the zeroth cyclic homology
of $A$, under Connes' map $B$. We refer to \cite{Ginzburg06} for a
wealth of examples related to the Ginzburg dg algebra. Our main
results state that $\Pi_n(A,c)$ is $n$-Calabi-Yau as a bimodule and
that the construction taking $(A,c)$ to $\Pi_n(A,c)$ is equivariant
under derived Morita equivalences and compatible with
localizations. In particular, we obtain that the Ginzburg dg algebra
is always $3$-Calabi-Yau. When informed of this fact, Michel Van den
Bergh provided an alternative proof \cite{VandenBergh07}, based on
noncommutative geometry. He has kindly made his proof available in the appendix
to this paper. The Calabi-Yau property of the Ginzburg dg algebra is
an important ingredient of Amiot's construction \cite{Amiot08a}
of the generalized cluster category associated to an algebra of
global dimension $\leq 2$ or a Jacobi-finite
quiver with potential. This construction in turn is an important
ingredient in the proof of the periodicity conjecture sketched in
section~8 of \cite{Keller08c}.

We compute deformed Calabi-Yau completions of most `homotopically
finitely presented dg categories' (\cf section~\ref{ss:CY-completion-of-htpfp-dg-categories}
for the definition) and use this to show that deformed
$3$-Calabi-Yau completions of algebras of global dimension at most
$2$ are quasi-isomorphic to Ginzburg dg algebras. A related
statement was proved independently by Ginzburg in \cite{Ginzburg08}.
As a corollary, we obtain that cluster-tilted algebras
\cite{BuanMarshReiten04} are Jacobian algebras of quivers with potentials, a
result that was proved independently by Buan-Iyama-Reiten-Smith
\cite{BuanIyamaReitenSmith08} using completely different methods.

As an application of the derived Morita equivariance of the construction of
deformed Calabi-Yau completions, we obtain a new construction of the
derived equivalence associated \cite{KellerYang09} to the mutation
of a quiver with potential \cite{DerksenWeymanZelevinsky08}. Our
approach also allows to generalize the mutation operation: For
a given quiver $Q$, each
tilting module over the path algebra $kQ$ yields a `generalized mutation'
of any quiver with potential of the form $(Q,W)$.

As an example of the localization theorem, we show that deleting
a vertex in a quiver with potential translates into a localization
of the associated Ginzburg algebra. In the case where the associated
Jacobian algebras are finite-dimensional, this localization then
yields a Calabi-Yau reduction \cite{IyamaYoshino08} of the
associated generalized cluster categories introduced by Amiot
\cite{Amiot08a}. A related result was recently obtained by
Amiot-Oppermann \cite{AmiotOppermann09}.

\subsection{Contents}
Each anti-involution $\tau: B \iso B\op$ of an algebra $B$ allows one to
define a preduality functor $M \mapsto \Hom_B(M,B)$ from the category of right $A$-modules
{\em to itself} by letting $B$ act on the target via $\tau$. The most important
example for us is the case where
$B=A\ten A\op$ and $\tau(a\ten a')=a'\ten a$. Bimodule duality is confusing
and the general context of an algebra with involution brings some
clarification. We develop the necessary material
in the setting of dg categories in section~\ref{s:preduality-functors}.

We then introduce and study the inverse dualizing complex of a homologically
smooth dg category in section~\ref{s:inverse-dualizing-complex}. We compute it
for (most) homotopically finitely presented dg categories (section~\ref{ss:htpfp-dg-categories})
and show that it behaves well under derived Morita equivalences
and localizations (proposition~\ref{prop:functoriality-of-Theta}). In particular,
homological smoothness and the Calabi-Yau property are preserved under
localizations.

We define $n$-Calabi-Yau completions in
section~\ref{s:Calabi-Yau-completions} and show that their
construction is compatible with derived Morita equivalences and
localizations
(proposition~\ref{prop:Morita-equivariance-CY-completion} and
theorem~\ref{thm:localization-CY-completion}). We show that Calabi-Yau
completions do have the Calabi-Yau property in
theorem~\ref{thm:CY-completion-is-CY}.  In
section~\ref{s:deformed-CY-completions}, we construct deformed
Calabi-Yau completions, prove that they have the Calabi-Yau property
(theorem~\ref{thm:defo-CY-completion-is-CY}), identify them with with
homotopy pushouts
(proposition~\ref{prop:defo-CY-completion-htpy-pushout}) and show
that their construction is compatible with derived Morita equivalences
and localizations (theorem~\ref{thm:localization-defo-CY-completion}).

After a reminder on Hochschild and cyclic homology of tensor
categories (section~\ref{ss:reminder-Hochschild-homology}), we recall
the definition of Ginzburg dg algebras in
section~\ref{ss:Ginzburg-dg-categories}. We interpret them as deformed
Calabi-Yau completions in
theorem~\ref{thm:CY-completion-is-Ginzburg-algebra}. In
section~\ref{ss:CY-completion-of-htpfp-dg-categories}, we observe that
deformed Calabi-Yau completions of homotopically finitely presented dg
categories are closely related to Ginzburg dg algebras.  We use this
in theorem~\ref{thm:3-CY-completion-of-2-dim-alg} to show that any
deformed $3$-Calabi-Yau completion of an algebra of global dimension
$\leq 2$ is a Ginzburg dg algebra. We apply this in
section~\ref{ss:Application-to-cluster-tilted-algebras} to show that
all cluster-tilted algebras are Jacobian algebras.

In the final section~\ref{s:particular-cases-localization-Morita-equivalence}, we
give two more applications of our general results to the study of mutations and of generalized
cluster categories. In corollary~\ref{cor:deleting-a-vertex-is-localization},
we show that deleting a vertex in a quiver $Q$ translates into a localization
of the Ginzburg algebra associated with any quiver with potential of the
form $(Q,W)$. In theorem~\ref{thm:localization-and-reduction} we prove that
in the associated generalized cluster categories, the localization yields
a Calabi-Yau reduction. We establish the link to Amiot-Oppermann's result
in section~\ref{ss:deleting-a-sink}. Finally, in section~\ref{ss:generalized-mutations},
we show that if $(Q,W)$ is a quiver with potential and $T$ any tilting
module for the path algebra $kQ$, there is an associated `generalized pre-mutation'
for $(Q,W)$. In particular, from the classical APR-tilts \cite{AuslanderPlatzeckReiten79},
one obtains the pre-mutation as defined in \cite{DerksenWeymanZelevinsky08}
and the associated derived equivalence of \cite{KellerYang09}.

In the appendix, Michel Van den Bergh uses non commutative differential geometry to
give an alternative proof of the fact that Ginzburg dg algebras have
the Calabi-Yau property.

\section*{Acknowledgments}

I thank Alastair King for helpful discussions and for suggesting the
name `Calabi-Yau completion'. I am grateful to Michel Van den Bergh
for agreeing to include his own proof of the Calabi-Yau property for
dg Ginzburg algebras as an appendix to this article. I am indebted
to Claire Amiot and Steffen Oppermann for letting me know about
their recent work and for pointing out a gap in a previous version
of the proof of theorem~\ref{thm:localization-and-reduction}. I
thank both of them as well as Ben Davison, 
Osamu Iyama, David Ploog, Rapha\"el Rouquier,
Michel Van den Bergh and Dong Yang for stimulating questions and
conversations.

\section{Preduality functors}
\label{s:preduality-functors}

\subsection{From involutions to preduality functors} \label{ss:from-involutions-to-weak-dualities}
Let $k$ be a commutative ring and $A$ an (associative, unital)
$k$-algebra. Let $\tau$ be an {\em involution on $A$}, \ie an
isomorphism from $A$ to the opposite algebra $A\op$ whose square
is the identity. Let $\Mod A$ denote the category of right
$A$-modules. If $M$ is a right $A$-module, the dual
\[
A^* = \Hom_A(M,A)
\]
becomes a left $A$-module via the left action of $A$ on itself, that
is to say, for an element $a\in A$ and an $A$-linear map $f$ from
$M$ to $A$, we define $af$ by
\[
(af)(m)= a f(m) \ko
\]
where $m$ runs through the elements of $M$. Now for any left
$A$-module $N$, we define the {\em conjugate right $A$-module}
$\ol{N}$ to be the abelian group $N$ endowed with the right action
by $A$ defined by
\[
n a = \tau(a) n \ko
\]
where $n$ is an element of $N$ and $a$ an element of $A$. In
particular, if $M$ is a right $A$-module, we obtain the {\em dual
right $A$-module}
\[
M^\vee =\ol{M^*}.
\]
The functor
\[
V: \Mod A \to (\Mod A)\op \]
taking $M$ to $VM=M^\vee$ together with
the natural transformation
\[
\phi:  M \to VVM
\]
given by evaluation
defines a {\em preduality functor} on the category $\Mod A$, \ie
the composition
\[
\xymatrix{ V \ar[r]^-{\phi V} & VVV \ar[r]^{V\phi} & V}
\]
is the identity. Equivalently, the map $f \mapsto \phi \circ f^\vee$
is a bijection
\[
\Hom_A(L, M^\vee) \iso \Hom_A(M, L^\vee)
\]
bifunctorial in the $A$-modules $L$ and $M$. Notice that
the left hand side is in canonical bijection with the set of {\em sesquilinear forms
on $L\times M$},
\ie maps
\[
s: L \times M \to A
\]
such that $s(l a , m) = \tau(a) s(l,m)$ and $s(l, ma)= s(l, m) a$ for
all $l\in L$, $m\in M$ and $a\in A$. Similarly, the right hand side
is in bijection with the set of sesquilinear forms on $M\times L$. The
bijection then corresponds to mapping a sesquilinear form $s$ to the
form $\tau \circ s \circ \sigma$, where $\sigma$ exchanges the two
factors of the product.

To say that $(V,\phi)$ is
preduality is also equivalent to saying that
the pair
\[
\xymatrix{
\Mod A \ar[d]<1ex>^{V} \\ (\Mod A)\op. \ar[u]<1ex>^{V\op} }
\]
together with the morphisms
\[
\phi: V V\op \to \id \mbox{ in } \Mod A \mbox{ and }
\phi: \id \to V\op V \mbox{ in } (\Mod A)\op
\]
is a pair of adjoint functors. So a preduality functor could also
be called a self-coadjoint functor.

If $(V,\phi)$ is a preduality functor, then so is $(V, -\phi)$.
An $A$-module $M$ is {\em reflexive for $V$}
is $\phi M$ is an isomorphism. For example, all finitely generated projective
$A$-modules are reflexive. A duality functor is a preduality functor
$(V,\phi)$ with invertible $\phi$. The restriction of a preduality functor
to the subcategory of reflexive objects is a duality functor.

\subsection{Extension of preduality functors to module categories}
\label{ss:preduality-functors-on-module-categories}
Now let $\ca$ be a $k$-category. By definition, the category $\Mod \ca$
of (right) $\ca$-modules is the category of $k$-linear functors
\[
M : \ca\op \to \Mod k.
\]
Suppose that $V$ is a preduality functor on $\ca$ and $\phi: \id \to VV$
the corresponding adjunction morphism. A {\em left $\ca$-module} is
a $k$-linear functor $N: \ca \to\Mod k$. Its {\em conjugate right module}
is the composition $\ol{N}=N\circ V$. The {\em dual left module $M^*$} of
a right $\ca$-module $M$ is the module given by
\[
X \mapsto \Hom_\ca(M, \ca(?,X)) \ko
\]
where $X$ runs through the objects of $\ca$. The {\em dual} (or, more
precisely, {\em $V$-dual}) of a right $\ca$-module $M$ is
\[
M^\vee = \ol{M^*}.
\]
It is given by
\[
X \mapsto \Hom_\ca(M, \ca(?,VX))\ko
\]
where $X$ runs through the objects of $\ca$. Let $L$ and $M$ be right
modules. Then the set
\[
\Hom(L, M^\vee)
\]
is in bijection with the set of {\em sesquilinear forms on $L\times M$}, \ie
the families of maps
\[
s_{X,Y} : LY \times MX \to \ca(X,VY)
\]
bifunctorial in the objects $X$ and $Y$ of $\ca$. By assumption
on $V$, we have a canonical bifunctorial bijection
\[
\theta: \ca(X,VY)\to \ca(Y, VX).
\]
By taking $s_{X,Y}$ to $\theta \circ s_{X,Y} \circ \sigma$,
where $\sigma$ exchanges the two factors, we obtain a bifunctorial
bijection
\[
\Hom(L, M^\vee) \to \Hom(M, L^\vee).
\]
It corresponds uniquely to a natural transformation
\[
\tilde{\phi} : M \to \tilde{V}\tilde{V}M\ko
\]
where $\tilde{V}M=M^\vee$. We conclude that
$(\tilde{V}, \tilde{\phi})$ is a preduality functor
on $\Mod \ca$. Notice that for a representable module
$\ca(?, X)$, we have a canonical isomorphism
\[
\tilde{V}(\ca(?,X)) \iso \ca(X,V?) \iso \ca(?, VX)
\]
and $\tilde{\phi}$ is induced by $\phi$ for such modules.
Thus the pair $(\tilde{V}, \tilde{\phi})$ is a preduality
functor which canonically extends $(V,\phi)$
from the subcategory of representable modules to all of $\Mod A$.
By abuse of notation, we will often write $(V,\phi)$ instead
of $(\tilde{V}, \tilde{\phi})$.

\subsection{Dg categories}
\label{ss:dg-categories}
Concerning dg categories, we follow the terminology and notations of \cite{Keller06d}.
Let us recall the most important points:
We fix a commutative ground ring $k$. Let $\ca$ be small dg $k$-category,
\ie  a small category enriched over the tensor category $\cc(k)$ of complexes
over $k$. A {\em dg $\ca$-module} is a dg functor
\[
M : \ca\op \to \cc_{dg}(k)
\]
with values in the dg category of complexes over $k$. In particular,
each object $X$ of $\ca$ gives rise to the {\em free module} (=representable
module)
$X^\wedge=\ca(?,X)$. The category of dg modules $\cc(\ca)$ has as
morphisms the morphisms of graded $\ca$-modules, homogeneous of
degree $0$ which commute with the differential. It is endowed with
a structure of Frobenius category whose conflations are the short
exact sequences of dg modules which split as sequences of graded modules.
The projective-injectives are the contractible dg modules. The associated
stable category is the homotopy category $\ch(\ca)$. It is
triangulated and its suspension functor takes a dg module $M$ to
$\Sigma M = M[1]$ whose underlying graded module has components
$(M[1](X))^p=M(X)^{p+1}$ and whose differential is $d_{M[1]}=-d_M$. The
category of {\em strictly perfect dg modules} is
the smallest subcategory of the Frobenius category $\cc(\ca)$
which contains the free dg modules and is stable under shifts, extensions
and passage to direct summands. The {\em derived category} $\cd(\ca)$
is the localization of the category $\ch(\ca)$ with respect to the
class of quasi-isomorphisms. It is a triangulated category with
suspension functor $\Sigma$. For each dg module $M$ and each
free module $X^\wedge$, we have a canonical isomorphism
\[
\Hom_{\cd(\ca)}(X^\wedge, \Sigma^n M) = H^n(M(X)).
\]
The derived category is compactly generated, in the sense
of \cite{Neeman99}, by the free modules $X^\wedge$,
$X\in\ca$. An object of $\cd(\ca)$ is defined to be {\em perfect} if
it is a compact object. The {\em perfect derived category $\per(\ca)$}
is the full subcategory of perfect objects of $\cd(\ca)$.
A dg functor $F:\ca\to\cb$ is a {\em Morita functor}
if restriction along $F$ is an equivalence from $\cd\cb$ to $\cd\ca$.
Equivalently, the total left derived
functor of the induction along $F$ is an equivalence. Still equivalently,
the morphisms
\[
\ca(X,Y) \to \cb(FX,FY)
\]
are quasi-ismorphisms for all $X$, $Y$ in $\ca$ and the objects
$F_*\ca(?,X) = \cb(?,FX)$ generate the perfect derived category $\per(\cb)$
as an idempotent complete triangulated category. In the localization
of the category of dg categories with respect to the class
of Morita functors, the set of morphisms from a dg category
$\ca$ to a dg category $\cb$ is in canonical bijection with
the set of isomorphism classes in $\cd(\ca\op\ten\cb)$
of dg $\ca$-$\cb$-bimodules $X$ such that $X(?,A)$ is
perfect as a dg $\cb$-module for each object $A$ of $\ca$,
\cf \cite{Toen07}. Two dg categories are {\em derived Morita
equivalent} if they become isomorphic in this localization.
Equivalently, they are linked by a chain of Morita functors.

\subsection{Preduality functors on dg categories}
\label{ss:preduality-on-dg-categories}
Let $\ca$ be a small dg category and $(V,\phi)$ a preduality dg functor on $\ca$.
Thus, $V$ is a dg functor $\ca\to\ca\op$ and $\phi : \id \to VV$
a natural transformation such that the map $f \mapsto V(f) \circ \phi$
is a bijection
\[
\ca(X,VY) \to \ca(Y, VX)
\]
for all objects $X$ and $Y$ of $\ca$. As in the case of the module
category over a $k$-linear category treated in
section~\ref{ss:preduality-functors-on-module-categories}, we have a
natural extension of $(V,\phi)$ to the category $\cc_{dg}(\ca)$ of
(right) dg $\ca$-modules.

Suppose from now on that $\ca$ is an exact dg category.
Recall that this means that the dg Yoneda functor
\[
\ca \to \cc_{dg}(\ca) \ko X \mapsto X^\wedge
\]
induces an equivalence onto a full subcategory which is stable
under shifts and under graded split extensions. In particular,
the category $\ca$ then has a canonical shift functor $\Sigma$
and each morphism $f$ of $Z^0\ca$ has a cone $C(f)$ whose image
under the Yoneda functor is the cone on $f^\wedge$.
In the underlying graded category $\ca_{gr}$, the cone
on a morphism $f$ from $X$ to $Y$ splits as $C(f)=Y\oplus \Sigma X$.
Let $i: Y \to C(f)$ be the inclusion and $h: X \to C(f)$ the
inclusion considered as a morphism of degree $-1$. Then the
pair $(i,h)$ is universal among the pairs consisting of a
closed morphism $j: Y \to Z$ and a morphism $l: X \to Z$ of
degree $-1$ such that $j\circ f=d(l)$.
\[
\xymatrix{X \ar[r]^f \ar@(ur,ul)[rr]^h \ar[dr]_l & Y \ar[d]^j \ar[r]^-i & C(f) \ar@{.>}[dl]\\
 & Z }
\]

Since $\ca$ is exact, the opposite dg category $\ca\op$ is also
exact. If $f: X \to Y$ is a closed morphism in $\ca$,
we can form its cone $C'(f)$ in $\ca\op$. In $\ca$, it is endowed with morphisms
$i': C(f) \to Y$ and $h': C'(f) \to Y$ such that $f\circ i'= d(h')$
and which are universal with this property. It follows that
$C'(f)$ splits as $\Sigma^{-1}Y \oplus X$ and that its
differential is given by the matrix
\[
\left[ \begin{array}{cc} -d_Y & f \\ 0 & d_X \end{array} \right].
\]
Thus, the shift $\Sigma C'(f)$ endowed with the canonical
morphisms $Y \to \Sigma C'(f)$ and $X \to \Sigma C'(f)$ is
uniquely isomorphic to the cone $C(-f)$ on the {\em opposite of $f$}.

Since $V$ is a dg functor, it preserves cones. So if $f: X\to Y$
is a closed morphism, we obtain a canonical isomorphism
\[
\Sigma VC_\ca(f) \iso C(-Vf)
\]
compatible with the (closed) inclusion  $i$ of $VX$ and the
inclusion $h$  (homogeneous of degree $-1$) of $VY$.

Let $n$ be an integer. Since $V$ is a dg functor from $\ca$ to
$\ca\op$, we have a canonical isomorphism
\[
V\Sigma^n \iso \Sigma^{-n} V.
\]
From $\phi$, we get a canonical isomorphism
\[
\psi: \id \to (\Sigma^n V)(\Sigma^n V)
\]
and it is not hard to check that $(\Sigma^n V, \psi)$ is
still a preduality dg functor.

Let $X$ be an object of $\ca$ and $f: X \to VX$ a closed
morphism. The morphism $f$ is {\em $(V,\phi)$-symmetric}
(respectively {\em antisymmetric}) if
\[
f= V(f)\circ \phi \mbox{ (respectively } f= - V(f)\circ \phi \mbox{ )}.
\]
The object $X$ is {\em reflexive} (respectively
{\em homotopy reflexive}) if $\phi : X \to VV X$
is an isomorphism (respectively if $H^0(\phi)$ is an isomorphism).
The analogue of the following proposition in a triangulated
setting is due to Balmer \cite[Theorem 1.6]{Balmer04}.

\begin{proposition}
\label{prop:symmetric-cone}
The cone on a $V$-antisymmetric closed morphism carries a canonical
$\Sigma V$-symmetric form. More precisely,
let $f: X \to X^\vee$ be a closed and $(V,\phi)$-antisymmetric morphism.
Let
\[
g: C(f) \to \Sigma V(C(f))
\]
be given by the matrix
\[
\left[ \begin{array}{cc} 0 & \id \\ \Sigma \phi & 0 \end{array} \right]:
VX\oplus \Sigma X \to \Sigma V VX \oplus \Sigma V \Sigma X.
\]
Then $g$ is a closed $(\Sigma V,\psi)$-symmetric morphism.
If $X$ is (homotopy) reflexive, then $g$ is invertible (up to homotopy).
\end{proposition}

\begin{proof} By the above discussion and the assumption that $f=-V(f)\circ \phi$,
the morphism $g$ is indeed well-defined and closed.
\[
\xymatrix{X \ar[r]^f \ar@(ur,ul)[rr]^h \ar[d]_\phi & VX \ar[d]^\id \ar[r]^i & C(f) \ar[d]^g\\
VVX \ar@(dr,dl)[rr]_h \ar[r]^{-Vf} & VX \ar[r]^i & \Sigma VC(f).}
\]
Clearly it is symmetric. We have
a morphism of graded split exact sequences
\[
\xymatrix{ 0 \ar[r] & VX \ar[d]^\id \ar[r] & C(f) \ar[d]_g \ar[r] & \Sigma X \ar[r] \ar[d]^{\Sigma \phi} & 0 \\
0 \ar[r] & VX \ar[r] & \Sigma VC(f) \ar[r] & \Sigma VVX \ar[r] & 0.}
\]
This implies that $C(f)$ is reflexive if $X$ is. By considering the
corresponding triangles in $H^0(\ca)$, we obtain that $H^0(g)$ is
an isomorphism if $H^0(\phi X)$ is an isomorphism.
\end{proof}

Now let $g: Y \to VY$ be a closed symmetric morphism and suppose that
$f:X\to Y$ is a closed morphism such that
\[
(Vf)\circ g \circ f=0.
\]
We then have a complex of closed morphisms
\[
\xymatrix@C=1.5cm{X \ar[r]^f & Y \ar[r]^{(Vf)\circ g}  & VX}
\]
and we can form its totalization, \ie the object $Z$
such that for $U$ in $\ca$, the complex $\ca(U,Z)$ is functorially
isomorphic to the total complex of
\[
\xymatrix@C=1.5cm{\ca(U,X) \ar[r]^{f_*} & \ca(U,Y) \ar[r]^{(Vf)_*\circ g_*}&  \ca(U,VX)} \ko
\]
where we think of $\ca(U,Y)$ as the zeroth column of the
double complex. The underlying graded object of $Z$ is isomorphic
to $\Sigma^{-1} VX\oplus Y \oplus \Sigma X$.
\begin{proposition} The graded morphism $h: Z \to VZ$ given by
$
\id_{VX}\oplus g \oplus \phi_X
$
is closed and $V$-symmetric. It is invertible (respectively invertible up to
homotopy) if $g$ is.
\end{proposition}

\begin{proof} We have a commutative diagram of complexes
\[
\xymatrix@C=1.6cm{ X \ar[d]_{\phi_X} \ar[r]^f & Y \ar[d]_g \ar[r]^{(VF)\circ g} & VX \ar[d]^{\id_X} \\
VVX  \ar[r]_{(VG)(VVf)} & VY \ar[r]_{VF} & VX.
}
\]
Therefore the morphism $h$ is closed. It is symmetric because $g$ and $\id_X \oplus \phi$
are symmetric.
\end{proof}

\subsection{Induction and preduality} \label{ss:induction} Let $\ca$ and $\cb$ be
two dg categories each endowed with a dg preduality functor denoted
by $(V,\phi)$. Let $F:\ca\to\cb$ be a dg functor. For a dg $\ca$-module
$M$, we denote by
\[
F_* M  \mbox{ or } M \ten_\ca \cb
\]
its induction along $F$. We assume that
we are given a morphism of dg functors
\[
FV \to VF.
\]
We wish to extend it to a compatibility morphism between induction
along $F$ and preduality with respect to $V$.

For each object $X$
of $\ca$, we have the representable left $\ca$-module $\ca(X,?)$.
Its image under induction along $F$ is $\cb(FX,?)$
and the predual of the image is
\[
\cb(FX,V?) \iso \cb(?,VFX).
\]
On the other hand, the predual of $\ca(X,?)$ is $\ca(?,VX)$ and
its image under induction is $\cb(?,FVX)$. Thus, the given morphism
$FV \to VF$ yields a natural transformation
\[
F_*(M^\vee) \to (F_* M)^\vee
\]
defined at first for representable and then for arbitrary dg
$\ca$-modules $M$.

If $M$ is a right dg $\ca$-module, then its dual
\[
M^*: X \mapsto \Hom_\ca(M, \ca(?,X))
\]
is a left dg $\ca$-module and we have a natural transformation
\[
F_* ( M^*) \to (F_* M)^*.
\]

By composing the natural transformations constructed so far, we
obtain, for each dg right $\ca$-module $M$, a natural transformation
\[
F_* (M^\vee) \to (F_* M^\vee)
\]
or, in the other notation,
\begin{equation} \label{eq:nat-trsf-induction-preduality}
M^\vee \ten_\ca \cb \to (M\ten_\ca \cb)^\vee.
\end{equation}

\begin{lemma} \label{lemma:induction-preduality}
\begin{itemize}
\item[a)] Under the natural transformation \ref{eq:nat-trsf-induction-preduality},
an element $f\ten b$ is sent to the map
\begin{equation}\label{eq:induction-preduality}
m\ten x \mapsto (-1)^{|f||b|}V(b) f(m) x.
\end{equation}
\item[b)] If the underlying graded $\ca$-module of $M$
is finitely generated projective, the transformation \ref{eq:nat-trsf-induction-preduality}
is invertible and its inverse sends an element $g$ to
\[
\sum m_i^* \ten V(g(m_i\ten \id)).
\]
where $\sum m_i \ten m_i^*$ is the Casimir element for $M$, \ie
the pre-image of the identity under the canonical
isomorphism
\[
M\ten_\ca \Hom_A(M,A) \to \Hom_\ca(M,M).
\]
\end{itemize}
\end{lemma}

\begin{proof} These are straightforward verifications.
\end{proof}

Let $\cd\ca$ denote the derived category of $\ca$.
We still denote by $M\mapsto M^\vee$ the total derived
functor of the duality functor and by $?\ten_\ca\cb$
the total derived functor $\cd\ca\to\cd\cb$ of the
induction functor.

\begin{lemma} \label{lemma:commutation-with-duality}
Suppose that $FV \to VF$ is a pointwise homotopy equivalence. Then
the morphism
\[
M^\vee\ten_A\cb \to (M\ten_\ca \cb)^\vee
\]
is a quasi-isomorphism for all perfect $M$. It is a quasi-isomorphism
for all $M$ if $\cb(F?,X)$ is perfect over $\ca$ for all $X$ in
$\cb$, for example if $F$ is a Morita functor.
\end{lemma}

\begin{proof} The canonical morphism
\[
\phi M: M^\vee\ten_A\cb \to (M\ten_\ca \cb)^\vee
\]
is a quasi-isomorphism for each
representable dg module $M=\ca(?,X)$, by the assumption on $F$ and $V$.
Since $\phi$ is a morphism between triangle functors, it
is still a quasi-isomorphism for each perfect dg
module $M$. Finally, if $\cb(F?,X)$ is perfect over $\ca$ for
all $X$ in $\cb$, then the derived tensor product $?\ten_\ca \cb$
preserves arbitrary products. Then $\phi$ is a morphism between
triangle functors taking arbitrary sums to products and hence
is a quasi-isomorphism for each object $M$ of $\cd \ca$.
\end{proof}

Now for a given right dg $\ca$-module $M$, we wish to study the dg $k$-module
\[
\Hom_\cb(M\ten_\ca \cb, (M\ten_\ca \cb)^\vee)
\]
(whose $n$th component is formed by the maps of graded $\cb$-modules
which are homogeneous of degree $n$). We can think of its elements
as sesquilinear forms on $M\ten_\ca\cb$. We have an isomorphism
\[
\Hom_\cb(M\ten_\ca \cb, (M\ten_\ca \cb)^\vee) =
\Hom_\ca(M, (M\ten_\ca \cb)^\vee)
\]
and the right hand side is the target of a natural transformation
with source
\[
(M\ten_\ca\cb)^\vee \ten_\ca M^*.
\]
Thus we obtain a natural transformation
\begin{equation} \label{eq:matrix-to-form}
M^\vee \ten_\ca \cb \ten_\ca M^* \to \Hom_\cb(M\ten_\ca \cb, (M\ten_\ca \cb)^\vee).
\end{equation}
Notice that the right hand side carries a natural involution,
namely the map taking $f$ to $f^\vee \circ \phi$. The left
hand side also carries a natural involution, namely the one
which on tensors of homogeneous elements is given by
\[
m_1 \ten b \ten m_2 \mapsto (-1)^{pq+pr+qr} m_2 \ten Vb \ten m_1\ko
\]
where $p, q, r$ are the degrees of $m_1$, $f$ and $m_2$, respectively.

\begin{lemma} The map \ref{eq:matrix-to-form} is strictly compatible
with these involutions.
\end{lemma}

\begin{proof} This is a straightforward verification.
\end{proof}

\section{The inverse dualizing complex}
\label{s:inverse-dualizing-complex}

\subsection{Duality for bimodules} \label{ss:duality-for-bimodules}
Let $k$ be a commutative
ring and $\ca$ a dg $k$-category. We may and will assume that $\ca$ is
cofibrant over $k$, \ie each morphism complex $\ca(X,Y)$ is cofibrant
in the category of dg $k$-modules. This always holds if $k$ is a field.
Let $\ca^e$ be the dg category $\ca\ten\ca\op$. We endow it with
the involution $V$ taking a pair of objects $(X,Y)$ to $(Y,X)$ and
given on morphisms by
\[
f\ten g \mapsto (-1)^{pq} g\ten f\ko
\]
where $f$ is of degree $p$ and $g$ of degree $q$. Note that $(V,\phi)$,
where the morphism $\phi$ is the identity, is a preduality on $\ca^e$
in the sense of section~\ref{ss:preduality-on-dg-categories}.

By a {\em bimodule} we mean a right dg module $M$ over $\ca^e$.
Via the morphism
\[
M\ten \ca^e = M\ten (\ca\ten\ca\op) \iso \ca\op\ten M \ten \ca
\]
taking $m\ten (a\ten b)$ to $(-1)^{|b|(|m|+|a|)}b\ten m \ten a$,
the right $\ca^e$-module structure
yields left and right $\ca$-module structures on $M$. The right module structure
on $\ca^e$ itself is given by the multiplication of $\ca^e$:
\[
(f\ten g)(f'\ten g') = f f' \ten g'g.
\]
So right multiplication yields the `inner' bimodule structure
on $\ca^e$, whereas the left $\ca^e$-module structure on $\ca^e$ yields
the `outer' bimodule structure.

As we have seen in section~\ref{ss:preduality-on-dg-categories},
from $(V,\phi)$, we obtain a natural preduality on the exact dg category
of dg $\ca^e$-modules which takes a dg module $M$ to the conjugate
$M^\vee$ of the dual $M^*$ defined by
\[
M^*: (X,Y) \mapsto \Hom_{\ca^e}(M, \ca^e(?,(X,Y))).
\]

\begin{lemma} \label{lemma:bimodule-induction-and-preduality}
Let $F: \ca\to\cb$ be a dg functor and $P$ an $\ca$-bimodule. We identify
$F_* P= P\ten_{\ca^e} \cb^e$ with $\cb\ten_\ca P \ten_\ca \cb$ via the
map $p \ten (x\ten y) \mapsto (-1)^{|y||p\ten x|} y\ten p \ten x$.
\begin{itemize}
\item[a)] The canonical morphism constructed in section~\ref{ss:preduality-on-dg-categories}
\[
\cb\ten_\ca P^\vee \ten_\ca \cb \to (\cb\ten_\ca P \ten_\ca \cb)^\vee
\]
takes $b_1 \ten f \ten b_2$ to the map
\[
x_1 \ten p \ten x_2 \mapsto \sum \pm b_1 f(p)_1 x_2 \ten x_1 f(p)_2 b_2 \ko
\]
where the sign is given by the Koszul sign rule and $f(p)=\sum f(p)_1 \ten f(p)_2$.
\item[b)] If the underlying graded module of $P$ is finitely generated
projective, the inverse
\[
(\cb\ten_\ca P \ten_\ca \cb)^\vee \to \cb\ten_\ca P^\vee \ten_\ca \cb
\]
of the morphism in a) takes a map $g$ to
\[
\sum \pm g(p_i)_1 \ten p_i^* \ten g(p_i)_2 \ko
\]
where the sign is given by the Koszul sign rule, we have
$g(p_i)=\sum g(p_i)_1\ten g(p_i)_2$ and $\sum p_i\ten p_i^*$
is the Casimir element for $P$.
\end{itemize}
\end{lemma}
\begin{proof} This is a special case of lemma~\ref{lemma:induction-preduality}.
\end{proof}

\subsection{Definition of the inverse dualizing complex}
\label{ss:definition-inverse-dualizing-complex}
As in section~\ref{ss:duality-for-bimodules}, we let $k$ be a
commutative ring and $\ca$ a dg $k$-category which is cofibrant
over $k$. We endow $\ca^e=\ca\ten\ca\op$ with the preduality
$(V,\phi)$ of section~\ref{ss:duality-for-bimodules}.
By $\ca$, we also denote the bimodule
\[
(X,Y) \mapsto \ca(X,Y).
\]
We define the {\em inverse dualizing complex $\Theta_\ca$} to
be any cofibrant replacement of the image of the bimodule $\ca$ under the total
derived functor of the preduality functor $M \mapsto M^\vee$ defined
in section~\ref{ss:duality-for-bimodules}.
Thus, if $\ca$ is given by a dg algebra $A$, then $\Theta_\ca$ is
a cofibrant replacement of
\[
\RHom_{A^e}(A,A^e)
\]
considered as an object of $\cd(A^e)$, \ie a right dg $A^e$-module, via the canonical
involution on $A^e$. Thus, the morphism set is computed using the
`inner' bimodule structure of $A^e$ and the right $A^e$-action on
$\Theta_A$ comes from the twisted right multiplication
\[
(a \ten b).(x\ten y) = V(x\ten y)(a\ten b) = (y\ten x) (a\ten b) = ya \ten bx
\]
which corresponds to the `outer' bimodule structure.
In this case, the homology $H^1\Theta_\ca$
is the space of outer double derivations of $A$, \ie the quotient
of the space of derivations of $A$ with values in $A^e$ by
the subspace of inner derivations. The inverse dualizing complex
owes its name to the following lemma. Let $\cd_{fd}(\ca)$ denote
the full subcategory of $\cd(\ca)$ formed by the dg modules $M$
such that each dg $k$-module $M(X)$, $X\in \ca$, is
perfect. If $k$ is a field and $\ca$ is
given by a dg algebra, this means that the sum $\sum_p \dim H^p(M)$ is finite.

\begin{lemma} \label{lemma:key-lemma} Suppose that $k$ is a field and $\ca$
is homologically smooth. For any dg module $L$ and any dg module $M$
in $\cd_{fd}(\ca)$, there is a canonical isomorphism
\[
\Hom_{\cd \ca}(L\ten_\ca \Theta_\ca, M) \iso D\Hom_{\cd \ca}(M,L) \ko
\]
where $D=\Hom_k(?,k)$. In particular, if $\Theta_\ca$ is isomorphic to
$\Sigma^{-n}\ca$ in $\cd(\ca^e)$, then $\cd_{fd}(\ca)$ is $n$-Calabi-Yau
as a triangulated category.
\end{lemma}

\begin{proof} This is a small variation on lemma~4.1 in \cite{Keller08d}.
\end{proof}

\subsection{Quivers, tensor categories, cyclic derivatives}
\label{ss:tensor-categories-cyclic-derivatives}
 In this section,
we collect preliminary material for the computation in
section~\ref{ss:htpfp-dg-categories}.
Let $Q$ be a {\em graded
$k$-quiver}, \ie $Q$ consists of a set of objects $Q_0$ and, for
all objects $x$ and $y$, a $\Z$-graded $k$-module $Q(x,y)$.
Let $\cR$ be the discrete $k$-category on $Q_0$: It has the set of objects
$Q_0$, each endomorphism algebra is isomorphic to $k$ and all
morphisms between different objects vanish. By abuse of notation, we
also denote by $Q$ the $\cR$-bimodule $(x,y) \mapsto Q(x,y)$. Recall that the tensor product
$L\ten_\cR M$ of a right by a left $\cR$-module is given by
\[
(L\ten_\cR M)(x,y) = \coprod_{z} L(z,y)\ten M(x,z) \ko
\]
where $z$ ranges over the objects of $\cR$. The {\em path category} of $Q$ is
the {\em tensor category $T_\cR(Q)$}: It has the set of objects $Q_0$ and
the bimodule of morphisms
\[
\cR \oplus Q \oplus (Q\ten_\cR Q) \oplus \ldots \quad
\]
with the natural composition. We put $\ca=T_\cR(Q)$.

Now assume that $Q$ is finitely generated and free as an
$\cR^e$-module. Fix a basis $\alpha_i$, $1\leq i\leq n$, of
$Q$ and let $\sum \alpha_i \ten \alpha_i^*$
be the {\em Casimir element} of the $\cR^e$-bimodule $Q$, \ie the
preimage of the identity under the canonical isomorphism
\[
Q \ten_{\cR^e} \Hom_{\cR^e}(Q, \cR^e) \to \Hom_{\cR^e}(Q,Q).
\]
The {\em cyclic
derivative with respect to $\alpha_i$} \cite{RotaSaganStein80} is the unique map
\[
\del_{\alpha_i} : T_\cR(Q) \to T_\cR(Q)
\]
taking a composition $\beta_1 \ldots \beta_s$ of elements of $Q$
to the sum
\[
\sum_j \alpha_i^*(\beta_j) \beta_{j+1} \ldots \beta_s \beta_1 \ldots \beta_{j-1}.
\]

\subsection{Computation for a homotopically finitely presented dg category}
\label{ss:htpfp-dg-categories} Let $k$ be a commutative ring and $Q$ a
graded $k$-quiver whose set of objects is finite and whose bimodule of
morphisms is finitely generated and projective over $k$. Let $\cR$ be
the $k$-category with the same objects as $Q$ and whose only non zero
morphisms are the scalar multiples of the identities. Let $\ca$ be a
dg category of the form $(T_\cR(Q), d)$, where $T_\cR(Q)$ is the
tensor dg category (\cf
section~\ref{ss:tensor-categories-cyclic-derivatives})
\[
\cR \oplus (Q \ten_\cR Q) \oplus \ldots \oplus (Q\ten_\cR \cdots
\ten_\cR Q) \oplus \ldots
\]
and the differential $d$ is such
that $Q$ admits a finite filtration
\begin{equation} \label{eq:filtration}
F_0 \subset F_1 \subset F_2 \subset \ldots \subset F_N = Q
\end{equation}
such that all $F_p$ have the same objects as $Q$, the bimodule of arrows
of $F_0$ vanishes and $d(F_p)$ is contained in $T_\cR(F_{p-1})$ for all $p \geq 1$.
As shown in \cite{ToenVaquie07}, \cf also \cite{Keller06d}, in the
Morita homotopy category of dg categories, the dg category $(T_\cR(Q), d)$ is
homotopically finitely presented and every homotopically finitely
presented dg category is a retract of such a dg category.
Our aim in this section is to compute the inverse dualizing complex
$\Theta_\ca$ for $\ca=(T_\cR(Q), d)$.
For this, we first need to construct a cofibrant resolution of $\ca$ over $\ca^e$.
Let $\tilde{\beta}$ be the unique bimodule derivation
\[
\ca \to \ca \ten_\cR Q \ten_\cR \ca
\]
which takes an element $v: x\to y$ of $Q$ to $\id_y \ten v \ten \id_x$.
Notice that $\tilde{\beta}$ vanishes on $\cR\subset \ca$.
If we have $n\geq 1$ and $a=v_1 \ldots v_n$ for elements
$v_i: x_{i} \to x_{i-1}$ of $Q$, we have
\[
\tilde{\beta}(a) = 1_{x_0}\ten v_1\ten v_2 \ldots v_n +
\sum_{i=2}^{n-1} v_1 \ldots v_{i-1} \ten v_i \ten v_{i+1} \ldots v_n
+ v_1 \ldots v_{n-1}\ten v_n \ten 1_{x_n}.
\]
Let us denote by
\[
\rho: \ca\ten_\cR \ca \ten_\cR \ca \to \ca\ten_\cR Q \ten_\cR \ca
\]
the $\ca$-bilinear extension of $\tilde{\beta}$.
Notice that $\rho$ is a retraction of the inclusion of
$\ca\ten_\cR Q \ten_\cR \ca$ into $\ca\ten_\cR \ca \ten_\cR \ca$.
Let $\delta$ be the composition
\[
\xymatrix{\ca\ten_\cR  Q \ten_\cR \ca \ar[r]^-{d} &
\ca\ten_\cR \ca \ten_\cR\ca \ar[r]^-{\rho} &
\ca\ten_\cR \ten Q \ten_\cR\ca}.
\]
\begin{proposition}
\label{prop:bimodule-resolution-of-A}
\begin{itemize}
\item[a)] We have $\delta^2=0$ and $\ca \ten_\cR Q \ten_\cR \ca$ endowed
with $\delta$ is a cofibrant dg bimodule.
\item[b)] The diagram
\[
\xymatrix{0 \ar[r] &
\ca \ten_\cR Q \ten_\cR \ca \ar[r]^-{\tilde{\alpha}}  &
\ca\ten_\cR \ca \ar[r] & \ca \ar[r] & 0 \ko
}
\]
where $\ca \ten_\cR Q \ten_\cR \ca$ is endowed with $\delta$
and
\[
\tilde{\alpha}(u\ten v\ten w)= uv\ten w - u\ten vw\ko
\]
is a complex of dg modules. The cone $\bp \ca$ over the morphism
\begin{equation} \label{eq:small-resolution-htfp}
\xymatrix{
\ca \ten_\cR Q \ten_\cR \ca \ar[r]^-{\tilde{\alpha}}  &
\ca\ten_\cR \ca  }
\end{equation}
is a cofibrant resolution of $\ca$ and is strictly perfect
(\cf section~\ref{ss:dg-categories}).
In particular, the dg category $\ca$ is homologically smooth.
\end{itemize}
\end{proposition}

\begin{remark} \label{rk:cofibrant-resolutions-of-A}
If instead of the finite filtration~\ref{eq:filtration}, we
have a countable exhaustive filtration $F_0 \subset F_1 \subset \ldots Q$
satisfying the same conditions, then the cone $\bp \ca$ of part b)
is still a cofibrant resolution of $\ca$ (but $\ca$ is no longer
homologically smooth in general).
\end{remark}

\begin{proof} a) Let us consider the commutator
$d \circ \rho - \rho \circ d$ as a graded map
from $\ca \ten_\cR \ca \ten_\cR \ca$ to itself.
Its restriction to
\[
\ca \iso \cR \ten_\cR \ca \ten_\cR \cR
\]
is a bimodule derivation. Since $\rho$ is bilinear,
the composition $\rho(d\circ \rho - \rho\circ d)$
still restricts to a bimodule derivation on $\ca$.
For $v\in Q$, we have
\[
\rho(d\circ \rho - \rho\circ d)(v) = \rho d(v) - \rho^2 d(v) =0.
\]
Thus, the composition $\rho(d\circ \rho - \rho\circ d)$ vanishes
on $Q$, thus on $\ca$ and thus on $\ca\ten_\cR \ca \ten_\cR \ca$.
It follows that we have
\[
\delta^2=\rho d \rho d = \rho^2 d^2 = 0.
\]
To check that $(\ca\ten_\cR \ca \ten_\cR \ca, \delta)$
is cofibrant it suffices to observe that $\delta$ takes
$\ca\ten_\cR F_p \ten_\cR \ca$ to $\ca\ten_\cR F_{p-1} \ten_\cR \ca$
for each $p\geq 1$ and that the subquotient is a finitely generated
free dg bimodule. Since the filtration by the $F_p$ is finite,
it also follows that $(\ca\ten_\cR \ca \ten_\cR \ca, \delta)$ is
perfect. Since $\cR$ is perfect over $\cR^e$ and
\[
\ca\ten_\cR \ca = \cR\ten_{\cR^e} \ca^e \ko
\]
it follows that the cone over
\[
\xymatrix{0 \ar[r] &
\ca \ten_\cR Q \ten_\cR \ca \ar[r]^-{\tilde{\alpha}}  &
\ca\ten_\cR \ca \ar[r] & 0 }
\]
is indeed cofibrant and perfect in $\cd(\ca^e)$.
\end{proof}

Let $\Theta_\ca=(\bp \ca)^\vee$ be the image under
the preduality functor $M \mapsto M^\vee$ defined
in section~\ref{ss:duality-for-bimodules} of the cofibrant
resolution $\bp \ca$ given by the cone over the morphism
\[
\xymatrix{
\ca \ten_\cR Q \ten_\cR \ca \ar[r]^-{\tilde{\alpha}}  &
\ca\ten_\cR \ca }
\]
of \ref{eq:small-resolution-htfp}. Since the cone
is stricly perfect, so is $\Theta_\ca$. In particular,
it is cofibrant and is therefore (homotopy equivalent to) the
inverse dualizing complex. Let us make $\Theta_\ca$ more
explicit. By definition, $\Sigma \Theta_\ca$ is
isomorphic to the cone of the induced morphism
\[
\xymatrix{
\Hom_{\ca^e}(\ca\ten_\cR \ca, \ca^e) \ar[r] &
\Hom_{\ca^e}(\ca \ten_\cR Q \ten_\cR \ca,\ca^e)
}
\]
endowed with the bimodule structure coming from the
`outer' structure on $\ca^e$.
Using lemma~\ref{lemma:bimodule-induction-and-preduality},
we obtain that $\Sigma\Theta$ is isomorphic to the cone over
the morphism of dg modules
\[
\xymatrix{
\ca\ten_\cR \cR^\vee \ten_\cR\ca \ar[r] & \ca\ten_\cR Q^\vee \ten_\cR \ca.
}
\]
which takes an element
$\id_x\ten\id_x^*\ten \id_x$ of $\ca\ten_\cR \cR^\vee \ten_\cR\ca$
to
\[
\id_x (\sum (-1)^{|\alpha_i|} \alpha_i^* \ten \alpha_i \ten \id_{x_i} -
\id_{x_i} \ten \alpha_i^* \ten \alpha_i) \id_x \ko
\]
where $\sum \id_x \ten \id_x^*$ is the Casimir element
of the $\cR^e$-module $\cR$ and $\sum \alpha_i \ten \alpha_i^*$ is the Casimir element
of the $\cR^e$-module $Q$ and $\alpha_i : x_i \to y_i$.
The differential of $\ca\ten_\cR \cR^\vee \ten_\cR$ is that
of the tensor product (where $\cR^\vee$ carries the zero differential).
To describe the differential of $\ca\ten_\cR Q^\vee \ten_\cR \ca$,
we consider $\ca\ten_\cR Q^\vee \ten_\cR\ca$ as a dg submodule
of the tensor algebra over $\cR$ of $Q\oplus Q^\vee$. Then
the differential of an element $\id_{x_i}\ten \alpha_i^* \ten \id_{y_i}$
equals the cyclic derivative (\cf section~\ref{ss:tensor-categories-cyclic-derivatives})
with respect to $\alpha_i$ of
\[
W= \sum_j (-1)^{|\alpha_j|} {\alpha_j^*} \, d(\alpha_j).
\]
This determines the differential because $\ca\ten_\cR Q^\vee \ten_\cR\ca$
is a dg $\ca$-bimodule whose underlying graded module is generated
by the elements $\id_{x_i}\ten \alpha_i^* \ten \id_{y_i}$.

\subsection{Compatibility with Morita functors and localizations}
\label{ss:Morita-equivariance-localization}
Keep the hypotheses of section~\ref{ss:definition-inverse-dualizing-complex}.
Let $\cb$ be another dg category and $F: \ca\to\cb$ a dg
functor. The dg functor $F$ is a {\em localization functor} if the (total left derived
functor of) induction along $F$ induces an equivalence
\[
(\cd\ca)/\cn \iso \cd\cb
\]
for some localizing subcategory $\cn$ of $\cd\ca$ (namely the kernel
of the induction functor). Equivalently, restriction along $F$
is an equivalence from $\cd\cb$ onto a full subcategory of $\cd\ca$
(whose inclusion admits a left adjoint given by the induction functor).
The localizations $F:\ca\to\cb$
such that the kernel $\cn$ of the induced functor $F_*: \cd(\ca)\to \cd(\cb)$
is compactly generated are precisely the dg quotients in the
sense of Drinfeld \cite{Drinfeld04} \cite{Keller99}.

\begin{proposition} \label{prop:functoriality-of-Theta} Assume that
$F:\ca\to\cb$ is a localization functor.
\begin{itemize}
\item[a)] The functor $F^e:\ca^e\to\cb^e$ induced by $F$ is still a
localization functor. It sends the bimodule $\ca$ to
the bimodule $\cb$.
\item[b)] The restriction $(F^e)^*$ along $F^e$ is monoidal
for the derived functors of the tensor products $\ten_\ca$ and $\ten_\cb$ (but does not
preserve the unit in general).
\item[c)]
If $\ca$ is homologically smooth, then
so is $\cb$ and the left derived functor of induction along
$F^e:\ca^e\to\cb^e$ sends $\Theta_\ca$ to $\Theta_\cb$. In particular,
for each dg $\cb$-module $L$, we have the projection formula
\begin{equation} \label{eq:projection-formula}
F_*((F^* L) \lten_\ca \Theta_\ca) \iso L\lten_\cb \Theta_\cb.
\end{equation}
\item[d)] If the dg category $\ca$ is homologically smooth and
$n$-Calabi-Yau as a bimodule for some integer $n$
(\cf section~\ref{ss:The-Calabi-Yau-property}), then $\cb$
has the same properties.
\item[e)] If $F$ is even a Morita functor, so is $F^e:\ca^e\to\cb^e$
and the induced equivalence $\cd(\ca^e) \to \cd(\cb^e)$
is naturally a monoidal functor for the derived functors of the tensor products $\ten_\ca$
and $\ten_\cb$. It  commutes
with the total derived functors of the preduality functors and
sends $\Theta_\ca$ to $\Theta_\cb$.
\end{itemize}
\end{proposition}

\begin{remark} If $A$ is an (ordinary) algebra and $A\to B$
a localization of $A$ in the sense that the induced functor
\[
\proj(A) \to \proj(B)
\]
between the categories of finitely generated projective modules
is a localization of categories, it may well happen that $A$ is
homologically smooth but $B$ is not. For example, if $A$ is the
path algebra of the quiver
\[
\xymatrix{1 \ar@<0.5ex>[r]^{\eps_2} \ar@<-0.5ex>[r]_{x_2} & 2 \ar@<0.5ex>[r]^{\eps_1} \ar@<-0.5ex>[r]_{x_1} & 3}
\]
over a field $k$, then $A$ is finite-dimensional and of global
dimension $2$ but its localization $B$ obtained by inverting $x_1$ and
$x_2$ is the $2\times 2$-matrix algebra over the algebra $k[\eps]/(\eps^2)$
of dual numbers. More generally, as shown in \cite{NeemanRanickiSchofield04},
every finitely presented $k$-algebra can be obtained in a similar way from
a finite-dimensional algebra of global dimension at most $2$. This is
not in contradiction with part a) of the lemma, because there,
we consider {\em derived} localizations. In fact, in our example, the
algebra $B$ is the zeroth homology of the dg quotient $\tilde{B}$
obtained from $A$ by inverting $x_1$ and $x_2$ and this generalizes to the setup of
\cite{NeemanRanickiSchofield04}.
\end{remark}

\begin{proof} Let us first describe the induction functor
$\cd(\ca^e) \to \cd(\cb^e)$ induced by $F$. For this, let us denote
by $X$ the $\ca$-$\cb$-bimodule $(A,B) \mapsto \cb(B,FA)$ and
by $X'$ the $\cb$-dual bimodule $(B,A) \mapsto \cb(FA,B)$, which
is isomorphic to $\RHom_\cb(X,\cb)$. Then the induction along
$F$ is isomorphic to the derived tensor product with $X$ and
the restriction along $F$ is isomorphic to the derived tensor
product with $X'$.
From the fact that $F$ is a localization functor, it follows
that the canonical morphism
\[
X'\lten_\ca X \to \cb
\]
is an isomorphism in $\cd(\cb^e)$. Moreover, since $X$ is
perfect over $\cb$, the canonical morphism
\[
X \lten_\cb X' \to \RHom_\cb(X,X)
\]
is an isomorphism in $\cd(\ca^e)$. The action of $\ca$ on $X$ yields
a bimodule morphism $\ca \to \RHom_\cb(X,X)$ and thus a morphism
\[
\ca \to X \lten_\cb X'
\]
in $\cd(\ca^e)$.
Now we can describe the induction functor $\cd(\ca^e) \to \cd(\cb^e)$: It is isomorphic
to
\[
M \mapsto X' \lten_\ca M \lten_\ca X.
\]
In particular, the bimodule
$M=\ca$ is sent to $X'\lten_\ca X \iso \cb$. The
restriction functor $\cd(\cb^e) \to \cd(\ca^e)$ is isomorphic
to
\[
N \mapsto X\lten_\cb N \lten_\cb X'.
\]
Since $X'\lten_\ca X$ is isomorphic to $\ca$, this shows part b):
the restriction functor is monoidal.
If we compose it with the induction functor, we find the
identity functor because $X'\lten_\ca X \iso \cb$. It follows
that the induction functor $\cd(\ca^e)\to \cd(\cb^e)$ is
a localization functor and sends $\ca$ to $\cb$, which is part a). If $\ca$
is homologically smooth, then $\ca$ is perfect in $\cd(\ca^e)$
and so its dual $\Theta_\ca$ is sent to the dual $\Theta_\cb$
of its image $\cb$, by lemma~\ref{lemma:commutation-with-duality}.
Thus, we have
\[
X^* \ten_\ca \Theta_\ca \ten_\ca X \iso \Theta_\cb.
\]
By applying $L\lten_\ca ?$ to this
isomorphism, we get the projection formula~\ref{eq:projection-formula}.
This ends the proof of c). Part d) is immediate from c) and a).

Let us prove e):
If $F$ is a Morita functor, the canonical morphism $\ca\to X\lten_\cb X'$
is also invertible and then the description of the induction functor
via $X$ and $X'$ shows that it is monoidal. The commutation of
the induction functor with the preduality functor follows from
lemma~\ref{lemma:commutation-with-duality}. Now the last assertion
follows from a).
\end{proof}

\section{Calabi-Yau completions}
\label{s:Calabi-Yau-completions}

\subsection{Definition and Morita equivariance}
\label{ss:definition-CY-completion-Morita-equivariance} Let $k$
be a commutative ring and $\ca$ a dg $k$-category whose morphism complexes
are cofibrant over $k$. Let $n$
be an integer and $\Theta=\Theta_\ca$ the inverse dualizing
complex of section~\ref{ss:definition-inverse-dualizing-complex}.
Put $\theta=\Sigma^{n-1}\Theta_\ca$. The {\em
$n$-Calabi-Yau completion of $\ca$} is the tensor dg category
\[
\Pi_n(\ca)=T_\ca(\theta) = \ca \oplus \theta \oplus (\theta \ten_\ca
\theta) \oplus \ldots .
\]
We also call it the {\em derived $n$-preprojective dg category of $\ca$}
(whence the notation $\Pi_n$). Notice that we have canonical inclusion
and projection functors
\[
\ca \to \Pi_n(\ca) \to \ca.
\]
Up to a quasi-isomorphism (canonical up to homotopy), it is independent of the
choice of cofibrant replacement made in the definition of $\Theta_\ca$.

\begin{proposition} \label{prop:Morita-equivariance-CY-completion}
Let $F:\ca\to \cb$ be a Morita functor. Then $F$ yields
a canonical Morita functor $\Pi_n(F): \Pi_n(\ca) \to \Pi_n(\cb)$ such
that we have a commutative diagram
\[
\xymatrix{\ca\ar[d]_F \ar[r] & \Pi_n(\ca)\ar[d]^{\Pi_n(F)} \ar[r] & \ca \ar[d]^F \\
\cb\ar[r] & \Pi_n(\cb) \ar[r] & \cb.
}
\]
\end{proposition}

\begin{proof} Let $F^e$ be the induced functor from $\ca^e$ to $\cb^e$
and denote by $F^{e*}$ the restriction along $F^e$.
By part e) of proposition~\ref{prop:functoriality-of-Theta}, we can find
a quasi-isomorphism $\phi: \Theta_\ca \to F^{e*}\Theta_\cb$ and
by part a), it induces quasi-isomorphisms between the (derived)
tensor powers
\[
\theta_\ca^{\ten_\ca n} \to F^*(\theta_\cb^{\ten_\cb n})
\]
for all $n\geq 1$. Thus, the pair $(F,\phi)$ yields a dg functor
\[
\Pi_n(F): T_\ca(\theta_\ca) \to T_\cb(\theta_\cb) \ko
\]
which is quasi fully faithful.  It remains to be shown that the image
generates the derived category of $\Pi_n(\cb)$.  Now clearly the image
contains all representable functors $\Pi_n(\cb)(?,FX)$ associated with
objects $FX$ in the image of $F$. But for an arbitrary object $M$ of
the derived category of $\Pi_n(\cb)$, we have
\[
\Hom (\Pi_n(\cb)(?,FX),M) = \Hom_{\cb}(\cb(?,FX), M|\cb)= M(FX).
\]
Now since $F$ is a Morita functor, the object $M$ vanishes iff
$M(FX)$ is acyclic for all $X$ in $\ca$. Thus, the right orthogonal
of the image of $\Pi_n(F)$ vanishes and so the image is all of the
derived category.
\end{proof}

\subsection{Morphisms between restrictions}
\label{ss:cy-completion-morphisms-between-restrictions}
We keep the
notations and assumptions of
section~\ref{ss:definition-CY-completion-Morita-equivariance}.
Let $i: \cd(\ca) \to \cd(\Pi_n(\ca))$ be the restriction along
the projection onto the first component $\Pi_n(\ca)\to \ca$.
\begin{lemma} \label{lemma:cy-completion-morphisms-between-restrictions}
Let $L$ and $M$ be in $\cd\ca$.
\begin{itemize}
\item[a)] We have a canonical isomorphism
\[
\RHom_{\Pi_n(\ca)}(iL, iM) = \RHom_\ca(L,M)\oplus \Sigma^{-n} \RHom_\ca(L\ten_\ca \Theta_\ca, M) \ko
\]
where $\Theta_\ca$ is the inverse dualizing complex (section~\ref{ss:definition-inverse-dualizing-complex}).
\item[b)] If $k$ is a field, $\ca$ is homologically smooth
and $M$ belongs to $\cd_{fd}(\ca)$ (\cf section~\ref{ss:definition-inverse-dualizing-complex}),
we have a canonical isomorphism
\[
\RHom_{\Pi_n(\ca)}(iL, iM) = \RHom_\ca(L,M)\oplus \Sigma^{-n} D\RHom_\ca(M,L) \ko
\]
where $D$ is the duality functor $\Hom_k(?,k)$.
\end{itemize}
\end{lemma}

\begin{proof} We may and will assume that $L$ is cofibrant over $\ca$. Then we have
an exact sequence of dg modules over $\Pi_n(\ca)=T_\ca(\theta)$
\begin{equation} \label{eq:resolution-A-module}
\xymatrix{0 \ar[r] & (iL)\ten_\ca \theta \ten_\ca T_\ca(\theta) \ar[r]^-\alpha &
(iL)\ten_\ca T_\ca(\theta) \ar[r]^-\beta &
iL \ar[r] & 0}\ko
\end{equation}
where $\alpha$ takes $l\ten x \ten u$ to $lx\ten u - l\ten xu$ and $\beta$ is the
multiplication of $iL$. Clearly the cone over $\alpha$ is a cobrant resolution
$\bp (iL)$ of $iL$ over $T_\ca(\theta)$. Since $\theta$ acts by zero in $iL$ and $iM$,
the morphism $\alpha$ induces zero in $\Hom_{T_\ca(\theta)}(?,iM)$. So we find
a canonical isomorphism in the derived category of $k$-modules
\[
\Hom_{T_\ca(\theta)}(\bp(iL), iM) = \Hom_\ca(L,M) \oplus \Sigma^{-1}\Hom_\ca(L\ten_\ca\theta, M).
\]
This implies part a). Part b) follows from part a) and Lemma~\ref{lemma:key-lemma}.
\end{proof}

\subsection{Compatibility with localizations}
\label{ss:compatibility-with-localizations}
We keep the
notations and assumptions of
section~\ref{ss:definition-CY-completion-Morita-equivariance}.
We say that a sequence of dg categories
\[
\xymatrix{
0 \ar[r] & \cn \ar[r]^G & \ca \ar[r]^F & \cb \ar[r] & 0
}
\]
is {\em exact} if the induced sequence
\[
\xymatrix{0\ar[r] & \cd(\cn) \ar[r]^{G_*} & \cd(\ca) \ar[r]^{F_*} & \cd(\cb) \ar[r] & 0}
\]
is exact, \ie the composition vanishes, $\cd(\cn)$ identifies with a full triangulated
subcategory of $\cd(\ca)$ and the triangle quotient of $\cd(\ca)$ by $\cd(\cn)$
identifies via $F_*$ with $\cd(\cb)$. In this case, the dg functor $F:\ca\to\cb$
is a localization in the sense of section~\ref{ss:Morita-equivariance-localization} (but
not each localization is obtained in this way as shown in
\cite{Keller94a}).

\begin{theorem}  \label{thm:localization-CY-completion}
Assume that $\ca$ is homologically smooth.
\begin{itemize}
\item[a)]
Let $F:\ca\to \cb$ be a localization functor. Then $F$ yields
a canonical localization functor $\Pi_n(F): \Pi_n(\ca) \to \Pi_n(\cb)$
such that we have a commutative diagram
\[
\xymatrix{\ca\ar[d]_F \ar[r] & \Pi_n(\ca)\ar[d]^{\Pi_n(F)} \ar[r] & \ca \ar[d]^F \\
\cb\ar[r] & \Pi_n(\cb) \ar[r] & \cb.
}
\]
\item[b)]
If we have an exact sequence of dg categories
\[
\xymatrix{
0 \ar[r] & \cn \ar[r]^G & \ca \ar[r]^F & \cb \ar[r] & 0}\ko
\]
then the kernel of the functor $\Pi_n(F)_*: \cd(\Pi_n(\ca)) \to \cd(\Pi_n(\cb))$
is the localizing subcategory generated by the objects $\Pi_n(\ca)(?,N)$, $N\in\cn$.
\end{itemize}
\end{theorem}

\begin{proof} We may and will assume that $F:\ca\to\cb$ is the
identity on the set of objects.
Let $(F^e)^*:\cc(\cb^e) \to \cc(\ca^e)$ be the restriction
functor. Let us put $\Theta_\cb'=(F^e)^*(\Theta_\cb)$. Notice
that for any objects $A,A'$ of $\ca$ (equivalently: $\cb$),
we have $\Theta'_\cb(A,A')=\Theta_\cb(A,A')$ and that in $\Theta'_\cb$,
the morphisms of $\ca$ act via $F:\ca\to\cb$.
According to part c) of
proposition~\ref{prop:functoriality-of-Theta}, we have
a canonical morphism of dg modules $\phi: \theta_\ca \to \theta_\cb'$
whose image under the induction along $F^e$ is invertible in $\cd(\cb^e)$.
The morphism $\phi$ yields morphisms of dg modules between the tensor powers
\[
\theta_\ca \ten_\ca \cdots \ten_\ca \theta_\ca \to
\theta'_\cb \ten_\ca \cdots \ten_\ca \theta_\cb' \to
\theta_\cb \ten_\cb \cdots \ten_\cb \theta_\cb.
\]
Thus, the pair $(F,\phi)$ yields a dg functor
$G: \Pi_n(\ca) \to \Pi_n(\cb)$. Clearly, $G$ is compatible
with the canonical inclusion and projection functors.
It remains to be shown that the restriction along $G$
is a fully faithful functor
\[
\cd(\Pi_n(\cb)) \to \cd(\Pi_n(\ca)).
\]
Let $L$ be a dg $\Pi_n(\cb)$-module. It is given by
its underlying dg $\cb$-module and a morphism of
dg $\cb$-modules
\[
\lambda : L \ten_\cb \theta_\cb \to L.
\]
The dg module $G^* L$ is given by the restriction of $L$
to $\ca$ and the morphism of dg $\ca$-modules deduced
from $\lambda$ and $\phi$
\[
\xymatrix{
L \ten_\ca \theta_\ca \ar[r]^{\id\ten \phi} & L\ten_\ca \theta_\cb
\ar[r]^{can} & L\ten_\cb \theta_\cb \ar[r]^-\lambda &  L.}
\]
Let us use this description of $G^*$ to show that it is
fully faithful. Let $L$ be a dg $\Pi_n(\cb)$-module. We may
and will assume that $L$ is cofibrant. Since $\Pi_n(\cb)$ is
cofibrant as a right dg $\cb$-module, the restriction of $L$
to $\cb$ is then cofibrant. We have an exact sequence of
cofibrant dg $\Pi_n(\cb)$-modules
\[
\xymatrix{
0 \ar[r] & L\ten_\cb \theta_\cb \ten_\cb T_\cb(\theta_\cb) \ar[r]^-{\alpha} &
L\ten_\cb T_\cb(\theta_\cb) \ar[r] & L \ar[r] & 0\ko
}
\]
where $\alpha(l\ten x\ten u) = lx\ten u - l\ten xu$.
This makes it clear that the cone over the morphism
\[
\xymatrix{
 L\ten_\cb \theta_\cb \ten_\cb T_\cb(\theta_\cb) \ar[r] &
L\ten_\cb T_\cb(\theta_\cb)
}
\]
is homotopy equivalent to $L$. Let $M$ be another dg $\Pi_n(\cb)$-module.
By applying $\Hom_\cb(?,M)$ to the above morphism, we obtain a morphism of
dg $k$-modules
\[
\Hom_\cb(L,M) \to \Hom_\cb(L\ten_\cb \theta_\cb, M)
\]
whose cone (shifted by one degree to the right)
computes morphisms from $L$ to $M$ in the derived category
of $\Pi_n(\cb)$. An analogous reasoning yields the morphisms
between $G^*L$ and $G^*M$ in the derived category of $\Pi_n(\ca)$.
Thus, to conclude that $G^*$ is fully faithful, it suffices
to check that for all $M$, $F^*$ induces bijections
\[
\Hom_{\cd(\cb)}(L, M) \to \Hom_{\cd(\ca)}(F^*L,  F^*M)
\]
and
\[
\Hom_{\cd(\cb)}(L\lten_\cb \theta_\cb, M) \to \Hom_{\cd(\ca)}(F^*(L)\lten_\ca \theta_\ca, F^*M).
\]
The first bijection follows from the full faithfulness of $F^*$.
The second one is a consequence of the full faithfulness of $F^*$
and of the projection formula~\ref{eq:projection-formula}. This ends
the proof of a). To prove b), it suffices to show that the image
of $\Pi_n(F)^*$ is exactly the full subcategory of the dg modules
over $\Pi_n(\ca)$ which are right orthogonal to all the representable
dg modules $\Pi_n(\ca)(?,N)$ for $N$ in $\cn$. We have
\[
\RHom_{\Pi_n(\ca)}(\Pi_n(\ca)(?,N), M)=\RHom_\ca(\ca(?,N), M) \ko
\]
which shows that if $M$ is in the image of $\Pi_n(F)^*$, it is
right orthogonal to the $\Pi_n(\ca)(?,N)$. Conversely, if $M$ satisfies
this condition, then the underlying $\ca$-module of $M$ is quasi-isomorphic
to $F^* L$ for some dg $\cb$-module $L$. The structural morphism
\[
M\ten_\ca \theta_\ca \to M
\]
then yields a morphism $F^*L \ten_\ca \theta_\ca \to F^*L$ hence
a morphism
\[
F_*(F^*L \ten_\ca \theta_\ca) \to L
\]
and thus by the projection formula~\ref{eq:projection-formula}, a
morphism
\[
L\ten_\cb \theta_\cb \to L
\]
Thus, $L$ carries a canonical structure of dg module over $\Pi_n(\cb)$ and
it is clear that $M$ is isomorphic to the image under $\Pi_n(F)^*$ of $L$
endowed with this structure.
\end{proof}

\subsection{The Calabi-Yau property} \label{ss:The-Calabi-Yau-property}
We keep the notations and assumptions
of section~\ref{ss:definition-CY-completion-Morita-equivariance}.
In particular, the symbol $n$ denotes a fixed integer.
On the category of dg $\ca^e$-modules, we consider the composition
$V_n$ of the preduality functor $V$ with the shift $\Sigma^n$.
It is part of a canonical preduality functor $(V_n, \phi_n)$
(by section~\ref{ss:preduality-on-dg-categories}).
We also use the notation $V_n$ for the derived functor of $V_n$.
Slightly modifying the terminology of Ginzburg and Kontsevich
(\cf Definition~3.2.3 of \cite{Ginzburg06}),
we say that the dg category $\ca$ is {\em $n$-Calabi-Yau
as a bimodule} if, in $\cd(\ca^e)$, there is an isomorphism
\[
f: \ca \to V_n \ca
\]
which is $(V_n, \phi_n)$-symmetric, \ie such that $V_n(f) \phi_n = f$.

\begin{theorem} \label{thm:CY-completion-is-CY}
If $\ca$ is homologically smooth,
its $n$-Calabi-Yau completion $\Pi_n(\ca)$ is homologically smooth and $n$-Calabi-Yau as
a bimodule.
\end{theorem}

\begin{proof} Let $\cb$ be the $n$-Calabi-Yau completion. We have
a short exact sequence of $\cb^e$-modules
\[
\xymatrix{ 0 \ar[r] & T_\ca(\theta)\ten_\ca \theta \ten_\ca T_\ca(\theta) \ar[r]^-\alpha &
T_\ca(\theta) \ten_\ca T_\ca(\theta) \ar[r] &
T_\ca(\theta) \ar[r] &
0\ko
}
\]
where the morphism $\alpha$ takes an element $f$ of $\theta(X,Y)$ to $1_Y\ten f - f\ten 1_X$
and the second map is composition. Thus, in the derived category of $\cb^e$,
the bimodule $T_\ca(\theta)$ is isomorphic to the cone on the morphism $\alpha$.
We deduce first that $T_\ca(\theta)$ is perfect as a bimodule: Indeed, the objects
\[
T_\ca(\theta)\ten_\ca \theta \ten_\ca T_\ca(\theta) = \theta \ten_{\ca^e} \cb^e
\quad\mbox{ and }\quad
T_\ca(\theta) \ten_\ca T_\ca(\theta) = \ca \ten_{\ca^e} \cb^e
\]
are perfect since they are induced from perfect $\ca^e$-modules (all
tensor products are also derived tensor products since $\cb^e$ is cofibrant
over $\ca^e$).

To prove the second part of the assertion, we first notice that $\theta$
is the $V_{n-1}$-dual of $\ca$. Since the bimodule $\ca$ is perfect, it is homotopically
$V_{n-1}$-reflexive and so, up to homotopy, $\ca$ is also the $V_{n-1}$-dual of
$\theta$. By lemma~\ref{lemma:commutation-with-duality}, for perfect
modules, the induction functor $?\ten_{\ca^e} \cb^e$
commutes with the preduality $V_n$ up to isomorphism in the derived
category. Thus, in $\cd(\cb^e)$, the objects
\[
\theta \ten_{\ca^e} \cb \mbox{ and } \ca \ten_{\ca^e} \cb
\]
are still $V_{n-1}$-dual to each other. So by proposition~\ref{prop:symmetric-cone},
in order to show that $\cb$ is $n$-Calabi-Yau as a bimodule, it suffices to show
that $\alpha$ is $V_{n-1}$-antisymmetric. Now as seen in section~\ref{ss:induction},
we have a natural homotopy equivalence
\[
V_{n-1}(\theta) \ten_{\ca^e} \cb^e \ten_{\ca^e} \theta^* \to \Hom_{\cb^e}(\theta\ten_{\ca^e} \cb^e, V_{n-1}(\theta\ten_{\ca^e} \cb^e)).
\]
The right hand side is quasi-isomorphic to the following dg $k$-modules:
\[
\Hom_{\ca^e}(\theta, V_{n-1}(\theta)\ten_{\ca^e}\cb^e) \iso
\Hom_{\ca^e}(\theta, \ca\ten_{\ca^e}\cb^e) =
\Hom_{\ca^e}(\theta, \cb\ten_\ca \cb) \ko
\]
where we use the fact that $\theta$ is cofibrant. So we get a natural
quasi-isomorphism
\[
V_{n-1}(\theta) \ten_{\ca^e} \cb^e \ten_{\ca^e} \theta^* \to \Hom_{\ca^e}(\theta, \cb\ten_\ca \cb).
\]
Let us lift the morphism $\lambda: x \mapsto 1 \ten x$ along this quasi-isomorphism:
Let $c$ be the Casimir
element in $\theta\ten_{\ca^e} \theta^*$, \ie the image of $1\in k$ under
the morphism
\[
k \to \Hom_{\ca^e}(\theta, \theta) \iso \theta\ten_{\ca^e} \theta^*.
\]
We let $\tilde{\lambda}$ be the image of $\id\ten c$ under the composition
\[
(V_{n-1}\theta)\ten_{\ca^e} (\ca\op\ten_k\theta) \ten_{\ca^e} \theta^* \to
(V_{n-1}\theta)\ten_{\ca^e} (\cb\op\ten_k \cb) \ten_{\ca^e} \theta^* .
\]
Then clearly $\tilde{\lambda}$ maps to $\lambda$ and the transpose conjugate
of $\tilde{\lambda}$ maps to $\rho: x \mapsto x \ten 1$. Since $\alpha$
equals $\rho-\lambda$, it follows that $\alpha$ is indeed $V_{n-1}$-antisymmetric.
\end{proof}

\section{Deformed Calabi-Yau completions}
\label{s:deformed-CY-completions}

\subsection{Construction and Calabi-Yau property} 
\label{ss:construction-CY-propery}
Let $k$ be a commutative ring and $\ca$ a dg
$k$-category such that $\ca(X,Y)$ is cofibrant as a dg $k$-module for
all objects $X$ and $Y$ of $\ca$.  We assume that $\ca$ is
homologically smooth. Let $\Theta$ be the inverse dualizing complex (\cf
section~\ref{ss:definition-inverse-dualizing-complex}), $n$ an integer, $\theta=\Sigma^{n-1}
\Theta$ and $\Pi_n(\ca)=T_\ca(\theta)$ the $n$-Calabi-Yau completion.
It is natural to deform $\Pi_n(\ca)$ by adding an $\ca$-bilinear
(super-)derivation $D$ of degree $1$ to its differential. Such a
derivation is determined by its restriction to the generating bimodule
$\theta$. It has to satisfy
\[
0=(d+D)^2 = d(D) + D^2.
\]
Since the right hand side is a degree $2$ derivation, it suffices to
check this identity on the generating bimodule $\theta$. Assume that
$D$ takes $\theta$ to $\ca\subset T_\ca(\theta)$. Then $D^2$ vanishes
and the condition reduces to $d(D)=0$. Thus, we see that each
closed bimodule morphism $c$ of degree $1$ from $\theta$ to $\ca$
gives rise to a `deformation'
\[
\Pi_n(\ca, c)
\]
of $\Pi_n(\ca)$, obtained by adding $c$ to the differential of $\Pi_n(\ca)$.
A standard argument shows that two homotopic morphisms $c$ and $c'$
yield quasi-isomorphic dg categories
$\Pi_n(\ca,c)$ and $\Pi_n(\ca, c')$. Thus, up to quasi-ismorphism, the
deformation $\Pi_n(\ca,c)$ only depends on the image of $c$ in the
derived category of bimodules (recall that $\theta$ is cofibrant).
Now notice that since the bimodule $\ca$ is perfect, we have the
following isomorphisms:
\begin{align*}
\Hom_{\cd(\ca^e)}(\Sigma^{n-1}\Theta, \Sigma \ca)  & =
\Hom_{\cd(\ca^e)}(\ca^\vee, \Sigma^{2-n} \ca) =
H^{2-n}(\ca \lten_{\ca^e} \ca^{\vee\vee}) \\
&= H^{2-n}(\ca \lten_{\ca^e} \ca) = \Tor^{\ca^e}_{n-2}(\ca, \ca)
= HH_{n-2}(\ca) \ko
\end{align*}
where $HH$ denotes Hochschild homology.

\begin{theorem} \label{thm:defo-CY-completion-is-CY}
The deformed $n$-Calabi-Yau completion $\Pi_n(\ca,c)$ associated
with an element $c$ of $HH_{n-2}(\ca)$ is homologically smooth
and $n$-Calabi-Yau.
\end{theorem}

\begin{proof} This is a variation on the proof of theorem~\ref{thm:CY-completion-is-CY}
where we have to take into account the new component of the differential of $T_\ca(\theta)$.
Let $\cb$ be the deformed $n$-Calabi-Yau completion. We still have
a short exact sequence of $\cb^e$-modules
\[
\xymatrix{ 0 \ar[r] & T_\ca(\theta)\ten_\ca \theta \ten_\ca T_\ca(\theta) \ar[r]^-\alpha &
T_\ca(\theta) \ten_\ca T_\ca(\theta) \ar[r] &
T_\ca(\theta) \ar[r] &
0\ko
}
\]
where the morphism $\alpha$ takes an element $f$ of $\theta(X,Y)$ to $1_Y\ten f - f\ten 1_X$
and the second map is composition. Notice that here the differentials of the tensor
algebras $T_\ca(\theta)$ are deformed but that the one of the middle factor $\theta$ on the
left is not! The map $\alpha$ is indeed compatible with the differential: For
an element $x$ of $\theta$, we have
\[
d(\alpha(x))= d(1\ten x -x \ten 1) = 1\ten (dx+cx) - (dx+cx)\ten 1 =1\ten dx - dx \ten 1\ko
\]
where the last equality holds because $cx$ belongs to $\ca$ and the tensor
product is over $\ca$. Now we can proceed as in the proof of theorem~\ref{thm:CY-completion-is-CY}.
We obtain that for arbitrary $c$, the dg category $\cb$ is smooth and $n$-Calabi-Yau.
\end{proof}

\begin{remark} The formulas in lemma~\ref{lemma:cy-completion-morphisms-between-restrictions}
remain true when we replace the Calabi-Yau completion $\Pi_n(\ca)$ with
the deformed Calabi-Yau completion $\Pi_n(\ca,c)$. Indeed, the sequence~\ref{eq:resolution-A-module}
in the proof of the lemma remains well-defined and exact
when replace $T_\ca(\theta)$ with $\Pi_n(\ca,c)$.
\end{remark}

\subsection{Deformed Calabi-Yau completions as homotopy pushouts}
\label{ss:Defo-CY-completions-via-htpy-pushout}
The slightly ad hoc construction of the deformed Calabi-Yau
completion given in section~\ref{ss:construction-CY-propery}
can be viewed more intrinsically as a homotopy pushout.
Let us explain this in more detail. Let $k$, $\ca$ and $\Theta$
be as in section~\ref{ss:construction-CY-propery} and let
$c$ be an element of $HH_{n-2}(\ca)$. We may lift $c$ to
a morphism of dg bimodules
\[
\tilde{c}: \Theta[n-2] \to \ca.
\]
This morphism extends uniquely to a morphism of dg categories
\[
[\id, \tilde{c}]: \Pi_{n-1}(\ca) \to \ca
\]
which is the identity on $\ca$ and given by $\tilde{c}$ on $\Theta[n-2]$.
We also have the projection
\[
[\id, 0] : \Pi_{n-1}(\ca) \to \ca.
\]
Now let $i: \ca \to \Pi_n(\ca,c)$ be the canonical inclusion.

\begin{proposition} \label{prop:defo-CY-completion-htpy-pushout}
The square
\[
\xymatrix{
\Pi_{n-1}(\ca) \ar[d]_{[\id, \tilde{c}]} \ar[r]^-{[\id, 0]} & \ca \ar[d]^i \\
\ca \ar[r]_-{i} & \Pi_n(\ca, c) }
\]
is a homotopy pushout square for the model category structure
on the category of dg categories introduced in \cite{Tabuada05a}.
\end{proposition}

Notice that the square is not commutative in the category
of dg categories. The proof will show in particular that it
becomes commutative in the homotopy category.

The proposition is a special case of the following general fact:
Let $\ca$ be any (small) dg category and $X$ a cofibrant $\ca$-bimodule.
Let $f: X \to \ca$ be a bimodule morphism. We also view $f$
as a morphism of degree $1$ from $X[1]$ to $\ca$. Let $T_A(X[1])$
denote the tensor category $T_\ca(X[1])$ whose differential
has been deformed using $f: X[1] \to \ca$ as an additional component.
Let the morphisms $[\id,f]$, $[\id,0]$ from $T_\ca(X)$ to $\ca$
and $i: \ca \to T_\ca(X[1],f)$ be defined analogously to
the above morphisms. Proposition~\ref{prop:defo-CY-completion-htpy-pushout}
is now clearly a special case of the following

\begin{proposition} The square
\[
\xymatrix{
T_\ca(X) \ar[d]_{[\id, f]} \ar[r]^-{[\id, 0]} & \ca \ar[d]^i \\
\ca \ar[r]_-{i} & T_\ca(X[1], f)}
\]
is a homotopy pushout square for the model category structure
on the category of dg categories introduced in \cite{Tabuada05a}.
\end{proposition}

\begin{proof} We may and will assume that $\ca$ is cofibrant
and that $X$ is cofibrant as a bimodule.
To compute the homotopy pushout of the angle
\[
\xymatrix{
T_\ca(X) \ar[d]_{[\id, f]} \ar[r]^-{[\id, 0]} & \ca  \\
\ca  & }
\]
it is then enough to replace the morphism $[\id,0]$ by a cofibration
and to compute the pushout in the category of dg categories.
To replace $[\id,0]$ by a homotopy pushout, we consider the
natural inclusion
\[
j: X \to IX
\]
of $X$ into the cone $IX$ over the identity of $X$. Clearly,
the morphism $[\id,0]$ factors as the cofibration $T_\ca(X) \to T_\ca(IX)$
followed by the trivial fibration $T_\ca(IX) \to \ca$. So to
compute the homotopy pushout, it is enough to compute the
homotopy pushout of the angle
\[
\xymatrix{
T_\ca(X) \ar[d]_{[\id, f]} \ar[r] & T_\ca(IX)  \\
\ca  & }
\]
We claim that this is given by the commutative squre
\[
\xymatrix{
T_\ca(X) \ar[d]_{[\id, f]} \ar[r] & T_\ca(IX) \ar[d]  \\
\ca \ar[r] & T_\ca(X[1],f).}
\]
Indeed, we have a pushout diagram of dg bimodules
\[
\xymatrix{
 X \ar[d]_{f} \ar[r]^j & IX \ar[d]  \\
\ca \ar[r] & \ca\oplus X[1]}
\]
where $\ca\oplus X[1]$ is endowed with the differential
of the mapping cone over $f$. 
Using this one easily checks that $T_\ca(X[1],f)$ has
the correct universal property.
\end{proof}

\subsection{Compatibility with Morita functors and localizations}
\label{ss:Morita-equivariance-localization-defo-CY-completion}
As in section~\ref{ss:construction-CY-propery}, let $n$ be an integer, $k$ a commutative
ring and $\ca$ a homologically smooth dg $k$-category such that $\ca(X,Y)$ is cofibrant
as a dg $k$-module for all objects $X$ and $Y$ of $\ca$. Consider
the deformed $n$-Calabi-Yau completion $\cb=\Pi_n(\ca,c)$ associated
with an element $c$ of $HH_{n-2}(\ca)$.

Now let $\cb$ be another dg $k$-category satisfying the same hypotheses
as $\ca$. Assume that we have a localization functor $F: \ca \to \cb$ and
let $c'$ be the element of $HH_{n-2}(\cb)$ obtained
as the image of $c$ under the map induced by $F$,
\cf \cite{Keller98}.

\begin{theorem} \label{thm:localization-defo-CY-completion}
\begin{itemize}
\item[a)] Under the above hypotheses,
there is a canonical localization functor
$G: \Pi_n(\ca, c) \to \Pi_n(\cb,c')$ such
that we have a commutative diagram
\[
\xymatrix{\ca\ar[d]_F \ar[r] & \Pi_n(\ca,c)\ar[d]^{G} \\
\cb\ar[r] & \Pi_n(\cb,c').
}
\]
The functor $G$ is a Morita functor if $F$ is.
\item[b)]
If we have an exact sequence of dg categories
(\cf section~\ref{ss:compatibility-with-localizations})
\[
\xymatrix{
0 \ar[r] & \cn \ar[r]^G & \ca \ar[r]^F & \cb \ar[r] & 0}
\]
then the kernel of the induced functor
\[
G_*: \cd(\Pi_n(\ca,c)) \to \cd(\Pi_n(\cb,c'))
\] is the localizing
subcategory generated by the dg modules $\Pi_n(\ca,c)(?,N)$,
where $N$ belongs to $\cn$.
\end{itemize}
\end{theorem}

\begin{proof} We have a commutative square of isomorphisms
\[
\xymatrix{H_{n-2}(\ca \lten_{\ca^e} \ca) \ar[r] \ar[d] & \Hom_{\ca^e}(\theta_\ca, \ca)  \ar[d]\\
H_{n-2}(\cb\lten_{\cb^e} \cb) \ar[r] & \Hom_{\cb^e}(\theta_\cb, \cb),}
\]
where the vertical arrows are induced by $F$. This yields a commutative
square in $\cd(\ca^e)$, where we also write $F^*$ for $(F^e)^*$,
\[
\xymatrix{\theta_\ca \ar[r]^{c} \ar[d]_\phi & \ca \ar[d]^F \\
F^*\theta_\cb \ar[r]^{F^*c'} & F^*\cb.}
\]
We would like to lift it to a strictly commutative square of dg modules.
We choose an arbitrary lift $\tilde{c}$ of $c$.
After replacing $\theta_\cb$ by a homotopy equivalent cofibrant dg module,
we may choose a dg module morphism $\tilde{c}': \theta_\cb \to \cb$ lifting
$c'$ such that $\tilde{c}'$ induces a split surjection
of graded $\cb^e$-modules. The same then holds for the morphism $F^* \tilde{c}'$
of dg $\ca^e$-modules. Therefore,
we can choose a lift $\tilde{\phi}$ of $\phi$ such that the square of
dg modules
\[
\xymatrix{\theta_\ca \ar[r]^{\tilde{c}} \ar[d]_{\tilde{\phi}} & \ca \ar[d]^F \\
F^*\theta_\cb \ar[r]^{F^*\tilde{c}'} & F^*\cb.}
\]
commutes strictly. As in the proof of theorem~\ref{thm:localization-CY-completion},
the morphisms $F$ and $\tilde{\phi}$ then induce a dg functor
\[
G: \Pi_n(\ca, c) \to \Pi_n(\cb, c').
\]
It remains to be checked that the restriction $G^*$ is a
fully faithful functor from $\cd(\Pi_n(\cb,c'))$ to
$\cd(\Pi_n(\ca,c))$. Let $L$ be a dg $\Pi_n(\cb, c')$-module.
It is given by its underlying dg $\cb$-module and
a morphism of {\em graded modules} homogeneous of degree~$0$
\[
\lambda: L\ten_{\cb} \theta_{\cb} \to L
\]
such that
\[
(d\lambda)(l\ten x) = l c'(x)
\]
for all $l$ in $L$ and $x$ in $\theta_{\cb}$. Suppose that $L$ is
cofibrant as a $\Pi_n(\cb,c')$-module. Since the underlying $\cb$-module
of $\Pi_n(\cb,c')$ is cofibrant (even with the deformed differential),
the underlying $\cb$-module of $L$ is cofibrant. We have an exact sequence of
cofibrant dg $\Pi_n(\cb)$-modules
\[
\xymatrix{
0 \ar[r] & L\ten_\cb \theta_\cb \ten_\cb T_\cb(\theta_\cb) \ar[r]^-{\alpha} &
L\ten_\cb T_\cb(\theta_\cb) \ar[r] & L \ar[r] & 0\ko
}
\]
where $\alpha(l\ten x\ten u) = lx\ten u - l\ten xu$. Notice that the map
$\alpha$ is a morphism of dg modules despite the deformation of the
differential on $T_\cb(\theta_\cb)$, analogously to what we have seen
in the proof of theorem~\ref{thm:defo-CY-completion-is-CY}.
The sequence shows that the cone over the morphism
\[
\xymatrix{
 L\ten_\cb \theta_\cb \ten_\cb T_\cb(\theta_\cb) \ar[r] &
L\ten_\cb T_\cb(\theta_\cb)
}
\]
is homotopy equivalent to $L$. Let $M$ be another dg $\Pi_n(\cb, c')$-module.
By applying $\Hom_\cb(?,M)$ to the above morphism, we obtain a morphism of
dg $k$-modules
\[
\Hom_\cb(L,M) \to \Hom_\cb(L\ten_\cb \theta_\cb, M)
\]
whose cone (shifted by one degree to the right)
computes morphisms from $L$ to $M$ in the derived category
of $\Pi_n(\cb, c')$. An analogous reasoning yields the morphisms
between $G^*L$ and $G^*M$ in the derived category of $\Pi_n(\ca, c)$.
Thus, to conclude that $G^*$ is fully faithful, it suffices
to check that for all $M$, the dg functor $F^*$ induces bijections
\[
\Hom_{\cd(\cb)}(L, M) \to \Hom_{\cd(\ca)}(F^*L,  F^*M)
\]
and
\[
\Hom_{\cd(\cb)}(L\lten_\cb \theta_\cb, M) \to \Hom_{\cd(\ca)}(F^*(L)\lten_\ca \theta_\ca, F^*M).
\]
As in the proof of theorem~\ref{thm:localization-CY-completion},
the first bijection follows from the full faithfulness of $F^*$ and
the second one is a consequence of the full faithfulness of $F^*$
and of the projection formula~\ref{eq:projection-formula}. This
ends the proof of a). The proof of b) is entirely analogous
to that of part b) of theorem~\ref{thm:localization-CY-completion}
and left to the reader.
\end{proof}

\section{Ginzburg dg categories}
\label{s:Ginzburg-dg-categories}

\subsection{Reminder on Hochschild and cyclic homology}
\label{ss:reminder-Hochschild-homology}
Let $k$ be a commutative ring and $Q$ a graded
$k$-quiver, \cf section~\ref{ss:tensor-categories-cyclic-derivatives}.
We put $\ca=T_\cR(Q)$.
The bimodule $\ca$ has the {\em small resolution}
\begin{equation} \label{eq:small-resolution}
\xymatrix{0 \ar[r] &
\ca \ten_\cR Q \ten_\cR \ca \ar[r]^-{\tilde{\alpha}}  &
\ca\ten_\cR \ca \ar[r] & \ca \ar[r] & 0 \ko
}
\end{equation}
where the map $\tilde{\alpha}$ takes a tensor $u\ten v\ten w$ to $uv\ten w - u\ten vw$
and the right hand map is composition. By tensoring this resolution with $\ca$
over $\ca^e$ we obtain the following complex which computes Hochschild
homology:
\[
\xymatrix{ 0 \ar[r] &  (Q\ten_\cR \ca)\ten_{\cR^e} \cR \ar[r]^-\alpha &
\ca\ten_{\cR^e} \cR \ar[r] & 0 }\ko
\]
where $\alpha$ takes a tensor $v\ten u$ with factors of degree $p$ and $q$
to $v u - (-1)^{pq} u v$.

Let $\tilde{\beta}$ be the unique bimodule derivation
\[
\ca \to \ca \ten_\cR Q \ten_\cR \ca
\]
which takes an element $v: x\to y$ of $Q$ to $\id_y \ten v \ten \id_x$.
If we have $n\geq 1$ and $a=v_1 \ldots v_n$ for elements
$v_i: x_{i} \to x_{i-1}$ of $Q$, we have
\[
\tilde{\beta}(a) = 1_{x_0}\ten v_1\ten v_2 \ldots v_n +
\sum_{i=2}^{n-1} v_1 \ldots v_{i-1} \ten v_i \ten v_{i+1} \ldots v_n
+ v_1 \ldots v_{n-1}\ten v_n \ten 1_{x_n}
\]
and
\[
\tilde{\alpha} \tilde{\beta} (a) = -1_{x_0} \ten a + a\ten 1_{x_n}.
\]
The map $\tilde{\beta}$ induces a (unique) map $\beta$ making the following
square commutative
\[
\xymatrix{
\ca \ar[r]^-{\tilde{\beta}} \ar[d]  &
 \ca \ten_\cR Q \ten_\cR \ca \ar[d] \\
\ca\ten_{\cR^e} \cR \ar[r]^-\beta &
(Q\ten_\cR \ca)\ten_{\cR^e}\cR \ko}
\]
where the left vertical map takes a path $a$ from $x$ to $y$
to $a\ten 1_x 1_y$ and the right vertical
map takes $a\ten v \ten b$ to $(-1)^{pq} (v\ten ba)\ten 1_x $,
where $a$ is of degree $p$  and $vb$ is of degree $q$. Note
that the tensor product $M\ten_{\cR^e} \cR$ of an $\cR$-bimodule
$M$ with $\cR$ over $\cR^e$ identifies with the quotient of $M$
by the dg submodule generated
by all differences $m 1_x - 1_x m$ for $m\in M$ and $x$ an object of
$\cR$. If we make this identification,
the map $\beta$ takes a path $v_1 \ldots v_n$ of $Q$
to the sum
\[
\sum_i \pm v_i \ten v_{i+1} v_{i+2} \ldots v_n v_1 \ldots v_{i-1} \ko
\]
where the sign is computed by the Koszul sign rule from the degrees of
the $v_j$. We clearly have $\alpha\circ \beta=0$.
The following complex is to be continued in a $2$-periodic fashion to the left
\begin{equation} \label{eq:small-cyclic-complex}
\xymatrix{\ldots \ar[r]^-\alpha &  \ca\ten_{\cR^e} \cR  \ar[r]^-\beta &
(Q\ten_\cR \ca)\ten_{\cR^e} \cR \ar[r]^-\alpha &
\ca\ten_{\cR^e} \cR \ar[r] & 0 }.
\end{equation}
It is the {\em small cyclic complex $C_{sm}(\ca)$} and computes cyclic
homology (\cf chapter~3 of \cite{Loday98}).
We sometimes consider its components as columns.
If $\ca=\cR$, cyclic homology is two-periodic, the
module $HC_1(\cR)$ vanishes and
$HC_0(\cR)$ is a sum of copies of $k$ indexed by $Q_0$.  If $k$
contains $\Q$, and $\ca$ is arbitrary, then the {\em reduced small
cyclic complex $C_{sm}(\ca)/C_{sm}(\cR)$} is quasi-isomorphic to
the quotient of its rightmost column by the image of $\alpha$, \ie to the
cokernel of the map
\[
\xymatrix{(Q\ten_\cR \ca)\ten_{\cR^e} \cR \ar[r]^-\alpha &\ca\ten_{\cR^e} \cR.}
\]
The inclusion of the subcomplex of the two rightmost terms
induces the canonical morphism from Hochschild to cyclic homology. The
corresponding quotient complex is isomorphic to the original complex
shifted by two degrees to the left. The short exact sequence thus
obtained induces the long exact sequence (known as the SBI-sequence)
\[
\xymatrix{
HH_n(\ca) \ar[r]^-I & HC_n(\ca) \ar[r]^-{S} &
HC_{n-2}(\ca) \ar[r]^-{B} & HH_{n-1}(\ca).}
\]
In particular, the rightmost arrow $\beta$ of the small cyclic complex
induces Connes' connecting map
\[
B: HC_n(\ca) \to HH_{n+1}(\ca).
\]
If the ring $k$ contains $\Q$ and the quiver $Q$ is concentrated
in degree $0$, then in the exact sequence
\[
HH_2(\ca) \to HC_2(\ca) \to HC_0(\ca) \to HH_1(\ca) \to HC_1(\ca) \ko
\]
the terms $HH_2(\ca)$ and $HC_1(\ca)$ vanish (as we see by
considering the small cyclic complex), the map $S$ induces
an isomorphism $HC_2(\ca) \iso HC_0(\cR)$, and the map
$B$ induces an isomorphism from the reduced zeroth
cyclic homology of $\ca$ to its first Hochschild homology.

\subsection{Ginzburg dg categories}
\label{ss:Ginzburg-dg-categories}
Let $Q$ be a graded $k$-quiver
such that the set of objects $Q_0$ is finite and $Q(x,y)$
is a finitely generated graded
projective $k$-module for all objects $x$ and $y$.
We fix an integer $n$ and a {\em superpotential of degree $n-3$}, \ie
an element $W$ in $(\ca \ten_{\cR^e}\cR) /\im \alpha$ of degree $n-3$. So $W$ is
a linear combination of cycles considered up to cyclic
permutation `with signs'. Notice
that $W$ need not be homogeneous with respect to the grading
by path length. We can view $W$ as an element in $HC_{n-3}(\ca)$
and if the ring $k$ contains $\Q$, every element of $HC_{n-3}(\ca)$
has such a representative. Let $\cR$ be the discrete category on
$Q_0$ and $Q^\vee$ the dual of the $\cR$-bimodule $Q$ over $\cR^e$
(endowed with the canonical involution).
Let $\sum v_i \ten v_i^*$ be the Casimir element of $Q\ten_{\cR^e} Q^\vee$,
\ie the element which, under the canonical isomorphism
\[
Q \ten_{\cR^e} Q^\vee \to \Hom_{\cR^e}(Q, Q) \ko
\]
corresponds to the identity of $Q$.

The {\em Ginzburg dg category $\Gamma_n(Q, W)$}, due to V.~Ginzburg
(section~4.2 of \cite{Ginzburg06})
for a quiver $Q$ concentrated in degree $0$ and $n=3$, is defined
as the tensor category over $\cR$ of the bimodule
\[
\tilde{Q}=Q  \oplus Q^\vee[n-2] \oplus \cR[n-1]
\]
endowed with the unique differential which
\begin{itemize}
\item[a)] vanishes on $Q$,
\item[b)] takes the element $v^*_i$ of $Q^\vee[n-2]$ to the cyclic derivative $\del_{v_i}W$
(\cf section~\ref{ss:tensor-categories-cyclic-derivatives}),
\item[c)] takes the element $\id_x$ of $\cR[n-1]$ to $(-1)^{n}\id_x (\sum [v_i, v_i^*]) \id_x$,
where $[,]$ denotes the supercommutator.
\end{itemize}

Let $\ca$ be the path category of $Q$ and $c=\beta(W)$ the image
of $W$ in
\[
HH_{n-2}(\ca) =\Tor^{\ca^e}_{n-2}(\ca, \ca).
\]
Thanks to the small resolution~\ref{eq:small-resolution}, the path
category $\ca$ is homologically smooth.
By theorem~\ref{thm:defo-CY-completion-is-CY}, the associated deformed
$n$-Calabi-Yau completion $\Pi_n(\ca, c)$ is homologically smooth and $n$-Calabi-Yau.

\begin{theorem} \label{thm:CY-completion-is-Ginzburg-algebra}
The deformed $n$-Calabi-Yau completion $\Pi_n(\ca, c)$ is quasi-isomorphic
to the Ginzburg dg category $\Gamma_n(\ca,W)$. In particular, the Ginzburg
dg category is homologically smooth and $n$-Calabi-Yau.
\end{theorem}

\begin{remark} If we use the theorem and
proposition~\ref{prop:defo-CY-completion-htpy-pushout}, 
we obtain that the Ginzburg dg category is given, up
to isomorphism in the homotopy category of dg categories
in the sense of \cite{Tabuada05a},
by the homotopy pushout square
\[
\xymatrix{
\Pi_{n-1}(\ca) \ar[d]_{[\id, \tilde{c}]} \ar[r]^-{[\id, 0]} & \ca \ar[d]^i \\
\ca \ar[r]_-{i} & \Gamma_n(\ca,W). }
\]
I thank Ben Davison \cite{Davison09} for suggesting this statement.
\end{remark}

\begin{proof} We first apply the computation of the inverse
dualizing complex of section~\ref{ss:htpfp-dg-categories} to
the special case where $\ca=T_\cR(Q)$ with $d=0$. We obtain that
the non deformed CY-completion is quasi-isomorphic to
the tensor category over $\cR$
of the bimodule $Q  \oplus Q^\vee[n-2] \oplus \cR[n-1]$
endowed with the unique differential which vanishes on $Q$ and $Q^\vee$ and
takes the element $\id_x$ of $\cR[n-1]$ to $(-1)^{n-2}\id_x (\sum [v_i, v_i^*]) \id_x$.
The deforming component of the differential of $\Pi_n(\ca,c)$
is the map $\theta \to \ca$ given by the contraction with $c=\beta(W)$ in
\[
\Sigma^{n-1}\RHom_{\ca^e}(\ca, \ca^e) \ten (\ca\lten_{\ca^e} \ca) \to \Sigma\ca.
\]
This last map identifies with
\[
\Sigma^{n-1} \Hom_{\ca^e}(P, \ca^e) \ten (P\ten_{\ca^e} \ca) \to \Sigma\ca \ko
\]
where $P$ is the cofibrant resolution of $\ca$ constructed in
proposition~\ref{prop:bimodule-resolution-of-A}.
The complex $P\ten_{\ca^e} \ca$ is isomorphic to
\[
\xymatrix{ 0 \ar[r] &  (Q\ten_\cR \ca)\ten_{\cR^e} \cR \ar[r]^-\alpha &
\ca\ten_{\cR^e} \cR \ar[r] & 0 }
\]
and $c$ lies in the subcomplex $(Q\ten_\cR \ca)\ten_{\cR^e} \cR$.
The complex $\Sigma^{n-1} \Hom_{\ca^e}(P, \ca^e)$ is isomorphic to
\[
 0 \to \ca \ten_\cR \ca \to \ca\ten_\cR Q^\vee \ten_\cR \ca \to 0.
\]
Therefore,
the deforming component of the differential vanishes on the left hand component
$\ca \ten_\cR \ca$.
Now it is clear that the deforming component of the differential vanishes
on $\cR[n-1]$ and
takes an element $v^*$ of $Q^\vee[n-2]$ to $(v^*\ten \id)\circ \beta(W)$.
For $v=v_i^*$, clearly this equals the cyclic derivative $\del_{v_i}W$.
\end{proof}

\subsection{Deformed Calabi-Yau completions of homotopically finitely presented dg categories}
\label{ss:CY-completion-of-htpfp-dg-categories} Let $k$ be a commutative ring and $Q$ a graded
$k$-quiver whose set of objects is finite and whose bimodule of morphisms
is finitely generated and projective over $k$. Let $\ca$ be a dg
category of the form $(T_\cR(Q), d)$, where the differential $d$ satisfies
the condition of section~\ref{ss:htpfp-dg-categories}.
Let $n$ be an integer, $Q^\vee=\Hom_{\cR^e}(Q,\cR^e)$ and
\[
\tilde{Q}=Q\oplus Q^\vee[n-2] \oplus \cR[n-1].
\]
Let $\sum \alpha_j \ten \alpha_j^*$ be the Casimir element of $Q$ and
let $W$ be the element
\[
W= \sum (-1)^{|\alpha_j|}\alpha_j^* d(\alpha_j)
\]
of $T_\cR(\tilde{Q})$. Let $W'$ be an element of $HC_{n-3}(\ca)$ and
$c\in HH_{n-2}(\ca)$ its image under Connes' map $B$.

\begin{proposition} \label{prop:CY-completion-htfp}
The deformed $n$-Calabi-Yau completion $\Pi_n(\ca, c)$
is isomorphic to the tensor category $T_\cR(\tilde{Q})$, endowed with the unique
differential $d$ such that for each $i$, we have
\[
d(\alpha_i)=\del_{\alpha_i^*} (W+W') \mbox{ and }
d(\alpha_i^*)= \del_{\alpha_i} (W+W')
\]
and for an object $x$ of $Q$, the element $\id_x$ of $\Sigma^{n-1}\cR$
is taken to
\[
d(\id_x) =(-1)^{n}\id_x (\sum [\alpha_i, \alpha_i^*]) \id_x
\]
where $[,]$ is the supercommutator.
\end{proposition}

\begin{proof} This follows from the description of the
inverse dualizing complex of $\ca$ in section~\ref{ss:htpfp-dg-categories}.
The details of the computation are similar to those in the proof of
Theorem~\ref{thm:CY-completion-is-Ginzburg-algebra} and left to the
reader.
\end{proof}

\subsection{$3$-Calabi-Yau completions of $2$-dimensional dg categories}
Let $k$ be a commutative ring and $\ca$ a dg category Morita equivalent
to $(T_\cR(V), d)$ for a graded $k$-quiver $V$ whose set of objects is finite
and whose bimodule of arrows is finitely generated free
over $k$ and concentrated in degrees $-1$ and $0$ (the differential
$d$ is arbitrary). The following proposition shows in particular that $\Pi_3(\ca)$ is
Morita-equivalent to a Ginzburg dg category.

Let $\cb$ be the path category $\cb=T_\cR(V^0\oplus (V^{-1})^\vee)$ of the
sum of the $0$th component of $V$ with the $\vee$-dual of $V^{-1}$ placed in
degree $0$. Let
$W$ be the class in $HC_0(\cb)$ of the element
\[
\sum_j v_j^* \, d(v_j) \ko
\]
where $\sum v_j \ten v_j^*$ is a Casimir element for $V^{-1}$.
Let $W'\in HC_0(\ca)$ and $c'\in HH_1(\ca)$ its image under Connes' map $B$.
For example we can have $W'=0$ and $c=0$.
\begin{proposition} \label{prop:3-CY-completion-dg-alg-dim-2}
The deformed $3$-Calabi-Yau completion $\Pi_3(\ca, c)$
is derived Morita-equi\-valent to the deformed $3$-Calabi-Yau completion
$\Pi_3(\cb, W+W')$ and thus to the Ginzburg algebra $\Gamma_3(V^0\oplus (V^{-1})^\vee,W+W')$.
\end{proposition}

\begin{proof} This is a special case of \ref{prop:CY-completion-htfp}.
\end{proof}

\subsection{$3$-CY completions of algebras of global dimension $2$}
\label{ss:CY-completion-alg-dim-2}
Let $k$ be a field and $A$ an algebra given as the quotient $kQ'/I$ of the
path algebra of a finite quiver $Q'$ by an ideal $I$ contained
in the square of the ideal $J$ generated by the arrows of $Q'$.
Assume that $A$ is of global dimension~$\leq 2$ (but not necessarily
of finite dimension over $k$). We construct a quiver $Q$ and a
superpotential $W$ as follows: Let $R$
be the union over all pairs of vertices $(i,j)$ of a set
of representatives of the vectors belonging to a basis of
\[
\Tor_2^A(S_j, DS_i) = e_j(I/(IJ+JI)) e_i \ko
\]
where $D=\Hom_k(?,k)$ and $S_i$ is the simple right module associated
with the vertex $i$.
We think of these representatives as `minimal relations' from $i$ to $j$,
\cf \cite{Bongartz83}.
For each such representative $r$, let $\rho_r$ be a new
arrow from $j$ to $i$. We define $Q$ to be obtained from $Q'$ by
adding all the arrows $\rho_r$. We define a potential by
\[
W=\sum_{r\in R} r \rho_r .
\]
Now let $W'\in HC_0(A)$ and $c\in HH_1(A,A)$ its image under
Connes' map $B$. Let $\tilde{W}'$ be an element of $HC_0(kQ)$
which lifts $W'$ along the canonical surjection $kQ \to kQ' \to A$
taking all arrows $\rho_r$ to $0$. For example, we can have $W'=0$
and $\tilde{W}'=0$.

\begin{theorem}
\label{thm:3-CY-completion-of-2-dim-alg}
The deformed $3$-Calabi-Yau completion $\Pi_3(A,c)$ is
quasi-isomorphic to the Ginz\-burg dg algebra $\Gamma_3(Q,W+\tilde{W}')$.
\end{theorem}

A very similar result was independently obtained by Ginzburg
\cite{Ginzburg08} in a slightly different setting.

\begin{proof} For each vertex $i$ of $A$ let $P_i$ be the indecomposable
projective $e_i A$. Let $\ca$ be the full subcategory of the module
category formed by the $P_i$. By induction, one constructs a graded $\cR$-bimodule
$V$ and a differential $d$ on $T_\cR(V)$ such that
\begin{itemize}
\item[1)] $V^n$ vanishes in degrees $n\geq 1$, $V^0$ is free with basis
$Q'$ and $V^{-1}$ is free with basis $R$;
\item[2)] the differential $d$ sends the basis element $r\in R$ of $V^{-1}$
to the element $r$ of $T_\cR(V^0)$;
\item[3)] for all $n\geq 1$, the differential $d$ maps $V^{-n-1}$ to
$T_n$ and induces an isomorphism from $V^{-n-1}$ onto $H^{-n}(T_n)$,
where $T_n$ denotes the dg category $T_\cR(V^0\oplus \cdots \oplus V^{-n})$.
\end{itemize}
Notice that a) the image $d(V)$ lies in the {\em square} of the ideal
generated by $V$ in $T_\cR(V)$ and that b) we have a canonical quasi-isomorphism
between $\cF=(T_\cR(V),d)$ and $\ca$.
The point a) implies that
we have isomorphisms
\[
V^{-n}(i,j) \cong \Tor^\cF_{1+n}(S_i,DS_j)
\]
for all $i$, $j$ and $n$ (thanks to remark~\ref{rk:cofibrant-resolutions-of-A},
we can use the bimodule resolution of part~b) of
proposition~\ref{prop:bimodule-resolution-of-A}).
The point b) implies that we have isomorphisms
\[
\Tor^\cF_{1+n}(S_i,DS_j) \cong \Tor^\ca_{1+n}(S_i,DS_j).
\]
Thus, we have $V^n=0$ for all $n$ different from $0$ and $-1$.
Now we can apply proposition~\ref{prop:3-CY-completion-dg-alg-dim-2}
to conclude.
\end{proof}

\subsection{Application to cluster-tilted algebras}
\label{ss:Application-to-cluster-tilted-algebras} Let $k$ be an algebraically
closed field.
If $A$ is a finite-dimensional $k$-algebra of finite global dimension,
its {\em generalized cluster-category} $\cc_A$ is defined as the full triangulated
subcategory of the triangle quotient
\[
\cd^b(A\oplus (DA)[-3])/\per(A \oplus (DA)[-3])
\]
generated by the image of the free module $A$, \cf \cite{Keller05}
and \cite{Amiot08a}. Here, the dg algebra $A\oplus (DA)[-3]$ is the
trivial extension of $A$ by the dg bimodule $(DA)[-3]$, where
$D=\Hom_k(?,k)$. In general, the category $\cc_A$ has infinite-dimensional
morphism spaces. As shown in \cite{Keller05}, if $A$ is the path
algebra of a quiver $Q$ without oriented cycles, then $\cc_A$ is
triangle equivalent to the cluster category $\cc_{Q}$ as defined in
\cite{BuanMarshReinekeReitenTodorov06}, \cf also
\cite{CalderoChapotonSchiffler06} for the case where $Q$ is Dynkin
of type $A$.

The {\em generalized cluster category $\cc_{(Q,W)}$} of a finite quiver $Q$ with
potential $W$ is defined as the triangle quotient
\[
\per(\Gamma_3(Q,W))/\cd^b(\Gamma_3(Q,W)\ko
\]
\cf \cite{Amiot08a}. In general, it has infinite-dimensional morphism spaces.
If $Q$ does not have oriented cycles (and so $W=0$), then $\cc_{(Q,0)}$
is equivalent to the cluster category $\cc_Q$, \cf \cite{Amiot08a}.
For arbitrary $(Q,W)$, the endomorphism algebra of the
image of the free module $\Gamma_3(Q,W)$ in $\cc_{(Q,W)}$
is isomorphic to the {\em Jacobian algebra} $H^0(\Gamma_3(Q,W))$.

Recall \cite{Keller07a} that a {\em tilting module} over an algebra $B$
is a $B$-module $T$ such that the total derived functor
of the tensor product by $T$ over the endomorphism
algebra $\End_B(T)$ is an equivalence
\[
\cd(\End_B(T)) \iso \cd(B).
\]
The endomorphism algebra $A$ of a tilting module $T$
over a hereditary algebra $B$ is of global dimension at most $2$.
A module $M$ is {\em basic} if each indecomposable module occurs
with multiplicity at most $1$ as a direct factor of $M$.
If $T$ is a basic tilting module over the path algebra $B=kQ''$ of a finite
quiver without oriented cycles, the endomorphism algebra $\tilde{A}$ of the
image of $T$ in $\cc_{Q''}$ is called the {\em cluster-tilted algebra}
associated with $T$, \cf \cite{BuanMarshReiten04}.

\begin{theorem} Let $A=kQ'/I$ be a $k$-algebra of global dimension at most $2$
as in section~\ref{ss:CY-completion-alg-dim-2} and define $(Q,W)$ as there.
Let $\Gamma=\Gamma_3(Q,W)$.
\begin{itemize}
\item[a)]
The category $\cc_{(Q,W)}$ is canonically triangle equivalent
to the cluster category $\cc_A$. The equivalence takes
$\Gamma$ to the image $\pi(A)$ of $A$ in $\cc_A$ and thus
induces an isomorphism from the Jacobian algebra $\cp(Q,W)$ onto
the endomorphism algebra $\tilde{A}$ of the image of $A$ in $\cc_A$.
\item[b)]
If $T$ is a basic tilting module over $kQ''$ for a quiver
without oriented cycles $Q''$ and $A$ is the
endomorphism algebra of $T$, then
$\cc_{(Q,W)}$ is triangle equivalent to $\cc_{Q''}$ by
an equivalence which takes $\Gamma$ to the image of $T$ in
$\cc_{Q''}$. Thus, the endomorphism algebra $\tilde{A}$ of $T$ in
$\cc_{Q''}$ is isomorphic to the Jacobian algebra $H^0(\Gamma)$.
\end{itemize}
\end{theorem}

The quiver of $\tilde{A}$ in part b) was first described by
Assem-Br\"ustle-Schiffler \cite{AssemBruestleSchiffler08}.
The fact that cluster-tilted algebras are Jacobian algebras
was independently proved by Buan-Iyama-Reiten-Smith
\cite{BuanIyamaReitenSmith08} using an entirely different method.

\begin{proof} a)
By theorem~\ref{thm:3-CY-completion-of-2-dim-alg}, the
$3$-Calabi-Yau completion $\Pi=\Pi_3(A)$ is quasi-isomor\-phic to $\Gamma=\Gamma_3(Q,W)$.
Thus we have an equivalence of triangulated categories
\[
\cc_{(Q,W)} \iso \per(\Pi)/\cd_{fd}(\Pi)
\]
taking the free module $\Gamma$ to $\Pi$. Moreover,
we have an equivalence of triangulated categories
\[
\per(\Pi)/\cd_{fd}(\Pi) \iso \cc_A
\]
taking the free module $\Pi$ to the image $\pi(A)$ of the free module $A$,
\cf the proof of theorem 7.1 in \cite{Keller05} or Lemmas 4.13 to 4.15
in \cite{Amiot08a}. The claim follows because $H^0(\Gamma)$ is isomorphic
to the endomorphism algebra of $\Gamma$ in $\cc_{(Q,W)}$ by theorem~3.6
of \cite{Amiot08a}.

b) If $A$ is the endomorphism algebra of $T$, then $A$ is derived
equivalent to the path algebra $kQ''$ and therefore $\cc_A$ is
equivalent to $\cc_{Q''}$. The claim now follows from part~a).
\end{proof}

\section{Particular cases of localization and Morita equivalence}
\label{s:particular-cases-localization-Morita-equivalence}

\subsection{Deleting a vertex is localization} \label{ss:deleting-a-vertex-is-localization}
Let $k$ be a field and $Q$ a finite quiver (possibly with oriented cycles).
Let $A$ be the path algebra $kQ$. Notice that $A$ may be of infinite dimension.
Let $i$ be a vertex of $Q$ and $e_i$ the associated idempotent.
Let $P_i=e_i A$ be the associated projective indecomposable. Let $\cn\subset \cd(A)$
be the localizing subcategory generated by $P_i$. Let $B=A/Ae_i A$. Notice
that $B$ is the path algebra of the quiver $Q'$ obtained from
$Q$ by deleting the vertex $i$ and all arrows
starting or ending at this vertex.

\begin{lemma} \label{lemma:deleting-a-vertex} The functor
\[
?\lten_A B: \cd(A) \to \cd(B)
\]
induces an equivalence from $\cd(A)/\cn$ onto $\cd(B)$. Thus, the
morphism $A \to B$ is a localization of dg categories (\cf section~\ref{ss:Morita-equivariance-localization}).
\end{lemma}

\begin{proof} Since $\cn$ is generated by a compact object, we
know (\cf for example \cite{Neeman99}) that for each object $X$
of $\cd(A)$, there is a triangle, unique up to unique isomorphism,
\begin{equation} \label{eq:localization-triangle}
X_\cn \to X \to X^{\cn^\perp} \to X_\cn
\end{equation}
with $X_\cn$ in $\cn$ and $X^{\cn^\perp}$ in the right
orthogonal subcategory $\cn^\perp$. Moreover, the projection
functor $\cd(A)\to \cd(A)/\cn$ induces an equivalence from $\cn^\perp$
onto $\cd(A)/\cn$. Let us compute the triangle~\ref{eq:localization-triangle}
for $X=P_j$, where $P_j=e_j A$ is the projective associated with a
vertex $j$ of $Q$. If we have $j=i$, the morphism $X_\cn \to X$ is the identity
of $P_i$. If we have $j\neq i$, let $\cm_j$ be the set of minimal elements
of the set of paths $p$ from $i$ to $j$, where we have $p\leq p'$
if $p'=pu$ for a path $u$ from $i$ to $i$. Then each morphism $P_i \to P_j$
uniquely factors through the morphism
\[
\bigoplus_{\cm_j} P_i \to P_j
\]
whose component associated with $p\in \cm_j$ is the left multiplication
by $p$. Moreover, this morphism is injective. It follows easily that it
induces a bijection
\[
\Hom_{\cd(A)}(\Sigma^m P_i, \bigoplus_{\cm_j} P_i) \to
\Hom_{\cd(A)}(\Sigma^m P_i, P_j)
\]
for each $m\in\Z$ and this implies that it induces a bijection
\[
\Hom_{\cd(A)}(N, \bigoplus_{\cm_j} P_i) \to
\Hom_{\cd(A)}(N, P_j)
\]
for each $N\in\cn$. It follows that the morphism
\[
\bigoplus_{\cm_j} P_i \to P_j
\]
is the universal morphism $X_\cn \to X$ for $X=P_j$. Therefore,
the object $P_j^{\cn^\perp}$ is the cokernel of
\[
\bigoplus_{\cm_j} P_i \to P_j.
\]
Now it is easy to check that for all vertices $j$ and
$l$, the morphism space
\[
\Hom_{\cd(A)}(P_j^{\cn^\perp}, \Sigma^m P_l^{\cn^\perp})
\]
vanishes for $m\neq 0$ and is canonically isomorphic to $e_l(A/Ae_i A) e_j$ for
$m=0$. This shows that the functor $?\lten_A (A/Ae_i A): \cd(A) \to \cd(A/Ae_i A)$
induces an equivalence from the subcategory of compact objects of
$\cd(A)/\cn$ onto the perfect derived category of $\cd(B)=\cd(A/Ae_i A)$.
Since this functor commutes with arbitrary coproducts, it does indeed
induce an equivalence from $\cd(A)/\cn$ onto $\cd(B)$.
\end{proof}

Recall that $Q$ is a finite quiver, possibly with oriented
cycles, $k$ is a field and $A$ is the path algebra $kQ$.
The quiver $Q'$ is obtained from $Q$ by deleting the vertex $i$
and all arrows starting or ending at $i$ and $B=A/Ae_i A$.
Now let $W$ be a potential on $Q$, \ie an element of $HC_0(A)$
and let $W'$ be the image of $W$ in $HC_0(B)$.
\begin{corollary} \label{cor:deleting-a-vertex-is-localization}
The canonical functor
\[
\Gamma_3(Q,W) \to \Gamma_3(Q',W')
\]
is a localization.
\end{corollary}

\begin{proof}
By the functoriality
of Connes' map $B$, the class $c'=B(W')$ is the image of $c=B(W)$
under the map $HH_1(A,A) \to HH_1(B,B)$ induced by $A\to B$.
By the localization theorem~\ref{thm:localization-defo-CY-completion}
and the above lemma~\ref{lemma:deleting-a-vertex},
we have an induced localization functor
\[
\Pi_3(A,c) \to \Pi_3(B,c')
\]
and by theorem~\ref{thm:CY-completion-is-Ginzburg-algebra}, this
yields a localization functor between the Ginzburg dg algebras.

\end{proof}

Let us put $\Gamma=\Gamma_3(Q,W)$ and $\Gamma'=\Gamma_3(Q',W')$.
Notice that in zeroth homology, the induced
morphism between the Jacobian algebras is the natural
quotient map
\[
\cp(Q,W) \to \cp(Q',W').
\]
Let us compare the generalized cluster categories
\[
\cc_{(Q,W)} = \per(\Gamma)/\cd_{fd}(\Gamma)
\]
and $\cc_{(Q',W')}$ under the assumption that these categories
have finite-dimensional morphism spaces. We refer to \cite{Amiot08a}
for a thorough analysis of this situation. Let
$\tilde{P}_i=e_i\Gamma$ and let $\ol{P}_i$ be the
image of $\tilde{P}_i$ under the projection functor
$\pi: \per(\Gamma) \to \cc$.

\begin{theorem} \label{thm:localization-and-reduction}
The triangulated category
$\cc_{(Q',W')}$ is triangle equivalent to
the Calabi-Yau reduction in the sense of
Iyama-Yoshino (section~4 of \cite{IyamaYoshino08})
of $\cc_{(Q,W)}$ at $\ol{P}_i$.
\end{theorem}

\begin{proof}  Let us put $\cc=\cc_{(Q,W)}$ and $\cc'=\cc_{(Q',W')}$.
Let $\cz$ be the full subcategory
of $\cc$ formed by the objects $M$ such that $\Ext^1(\ol{P}_i,M)$
vanishes. By definition, the Calabi-Yau reduction at $\ol{P}_i$ is the
quotient $\cz/(\ol{P}_i)$ of $\cz$ by the ideal of morphisms
factoring through a finite direct sum of copies of $\ol{P}_i$. To construct a functor from $\cz$ to
$\cc'$, we consider the fundamental domain $\cF\subset \per(\Gamma)$
as defined in section~2.2 of \cite{Amiot08a}. Thus, the subcategory
$\cF$ can be described as the full subcategory
\[
\per(\Gamma) \cap \cd_{\leq 0} \cap \mbox{}^\perp(\cd_{\leq -2}) \ko
\]
where $\cd_{\leq 0}$ is the left aisle of the canonical $t$-structure
on $\cd(\Gamma)$. Alternatively, the subcategory $\cF$ can be described as the full
subcategory whose objects are the cones on morphisms between
objects of the closure $\add(\Gamma)$ of the free module $\Gamma$ under finite
direct sums and direct factors. We know from
[loc. cit.] that the projection induces a $k$-linear equivalence
$\cF \iso \cc$. Now we consider the composition
\[
\cz \subset \cc \iso \cF \to \cF' \iso \cc'
\]
where $\cF'$ is the fundamental domain for $\cc'$.
Let us denote this functor by $F$.
Its restriction to the full subcategory $\ct$ whose objects
are the $\ol{P}_j$ associated with all vertices $j$
identifies with the canonical projection
functor
\[
\cp(Q,W) \to \cp(Q',W').
\]
In particular, since $\cp(Q,W)$ is isomorphic to $\ct$
by theorem~2.1 of \cite{Amiot08a}, the restriction induces an equivalence
\[
\ct/(\ol{P}_i) \to \ct'
\]
where $\ct'\subset\cc'$ is the full subcategory
of the $\ol{P}_j$, $j\neq i$. We will show below that the
functor $\cz/(\ol{P}_i) \to \cc'$ induced by $F$ is
naturally a triangle functor. Since this triangle functor
induces an equivalence between the cluster-tilting
subcategories
\[
\ct/(\ol{P}_i) \to \ct'\ko
\]
it is itself an equivalence by lemma~4.5 of
\cite{KellerReiten08}.

It remains to be shown that the functor $\ol{F}:\cz/(\ol{P}_i) \to \cc'$
induced by $F$ is naturally a triangle functor. Let
$q: \cc \to \cF$ be a $k$-linear quasi-inverse of the projection $\cF \to \cc$.
Let
\[
\xymatrix{ X \ar[r]^u & Y \ar[r]^v & Z \ar[r]^w & \Sigma X}
\]
be a triangle of $\cc$ such that $X$, $Y$ and $Z$ lie in $\cz$.
Notice that $v$ induces a surjection
\[
\cc(\ol{P}_i, Y) \to \cc(\ol{P}_i, Z).
\]
Form a triangle in $\per(\Gamma)$
\[
\xymatrix{X' \ar[r] & q(Y) \ar[r]^{q(v)} & q(Z) \ar[r] & \Sigma X'}.
\]

\noindent{\em Claim: The object $\tau_{\leq 0} X'$ lies in $\cF$. The object
$q(X)$ is isomorphic to $\tau_{\leq 0} X'$ by an isomorphism canonical up to a morphism factoring
through $q(Y)$. Moreover, the image of the morphism $\tau_{\leq 0} X' \to X'$
under the composed functor $\per(\Gamma) \to \per(\Gamma')\to \cc'$ is invertible.}

Indeed, from the triangle
\[
\xymatrix{\Sigma^{-1} q(Z) \ar[r] & X' \ar[r] & q(Y) \ar[r]^{q(v)} & q(Z) } \ko
\]
we see that $X'$ is left orthogonal to $\cd_{\leq -2}$. If $M$ belongs to $\cd_{\leq 0}$,
we have, using the Calabi-Yau property and the fact that $\tau_{>0} X'$ belongs
to $\cd_{fd}(\Gamma)$, the isomorphisms
\[
\Hom(\Sigma^{-1} \tau_{>0} X', \Sigma^2 M) = D\Hom(\Sigma^{-1} M, \Sigma^{-1} \tau_{>0} X')=0.
\]
Now from the triangle
\[
\xymatrix{\Sigma^{-1} \tau_{>0}X' \ar[r] & \tau_{\leq 0} X' \ar[r] & X' \ar[r]  & \tau_{>0} X'} \ko
\]
we see that $\tau_{\leq 0} X'$ belongs to $\mbox{}^\perp \cd_{\leq -2}$ and of course,
it belongs to $\cd_{\leq 0}$. Thus, it belongs to $\cF$. By our assumption, the object $\tau_{>0} X'$ has finite-dimensional
homology. Thus, the image of $\tau_{\leq 0} X'$ in $\cc$ is isomorphic to $\pi(X')$. By the
uniqueness of the triangle on the morphism $v: Y \to Z$, we obtain that $X$
is isomorphic to $\pi(\tau_{\leq 0} X')$ by a morphism canonical up to a morphism factoring through $Y$.
Thus, since $\tau_{\leq 0} X'$ belongs to $\cF$, the object $q(X)$ is isomorphic to
$\tau_{\leq 0} X'$ by an isomorphism canonical up to a morphism factoring through $q(Y)$.
Finally, the homology of $\tau_{>0} X'$ is concentrated
in degree $1$, and we have an exact sequence
\[
H^0(q(Y)) \to H^0(q(X)) \to H^1(\tau_{>0} X') \to 0.
\]
In particular, we have an exact sequence
\[
\Hom(\tilde{P}_i, q(Y)) \to \Hom(\tilde{P}_i, q(Z)) \to \Hom(\tilde{P}_i, \tau_{>0} X') \to 0.
\]
Since $\Hom(\tilde{P}_i, q(U))$ is isomorphic to $\Hom_{\cc}(\ol{P}_i, U)$ for each $U$ in $\cc$,
it follows that $\tau_{>0} X'$ is right orthogonal to $\Sigma^m \tilde{P}_i$ for all $m\in\Z$.
Thus it is right orthogonal to the kernel of the localization functor $L:\cd\Gamma\to \cd\Gamma'$.
Therefore, for each object $M$ of $\cd\Gamma$, the localization functor induces
a bijection
\[
\Hom(M, \tau_{>0} X') \to \Hom(LM, L\tau_{>0} X').
\]
If, for $M$, we take the objects $\Sigma^m \tilde{P}_j$ associated with the
vertices of $Q$, we obtain that $L\tau_{>0}X'$
has its homology of finite total dimension.
This implies the last part of the claim.

Now let us show that the functor $\ol{F}:\cz/(\ol{P}_i) \to \cc'$
induced by $F$ is naturally a triangle functor. In any triangulated
category, by default, we denote the suspension functor by $\Sigma$ and a
quasi-inverse of $\Sigma$ by $\Omega$. However, we denote
the desuspension functor of the `reduced' category $\ol{\cz}=\cz/(\ol{P}_i)$ by $\Omega_r$.
We will construct a natural isomorphism $\phi: \Omega \ol{F} \iso \ol{F} \Omega_r$ and show that the pair
$(\ol{F}, \phi)$ transforms triangles into triangles. Let $Z$ be an
object of $\cz$ and $P \to Z$ a right approximation of $Z$ by
$\add(\ol{P}_i)$. Form the triangle
\[
\Omega_r Z \to P \to Z \to \Sigma \Omega_r Z
\]
of $\cc$. The object $\Omega_r Z$ still belongs to $\cz$ and its
image in $\ol{\cz}$ is the desuspension of the image of $Z$. Now
form a triangle of $\per(\Gamma)$
\[
O \to q(P) \to q(Z) \to \Sigma O.
\]
Let us denote the composition of the localization functor
$L: \per(\Gamma)\to\per(\Gamma')$ with the projection $\per(\Gamma')\to \cc'$
by $L': \per(\Gamma) \to \cc'$.
By the claim, we have an isomorphism
\[
q(\Omega_r Z) \iso \tau_{\leq 0} O
\]
canonical up to a morphism factoring through $q(P)$ and
the morphism $L'\tau_{\leq 0}O \to L'O$ is invertible. The triangle
\[
\Omega q(Z) \to O \to q(P) \to q(Z)
\]
and the triangle structure on $L'$ yield an
isomorphism $\Omega L' q(Z) \to L'\Omega q(Z) \to L' O$.
Thus, we obtain a canonical composed isomorphism
\[
\Omega \ol{F} Z =\Omega L' q(Z) \iso L' \Omega q(Z) \iso L'O \liso L'(\tau_{\leq 0} O) \liso L'q\Omega_r(Z) = \ol{F} \Omega_r(Z)
\]
and we define $\phi(Z)$ to be this isomorphism. One checks that $\phi(Z)$
is natural in the object $Z$ of $\ol{\cz}$. Now let a
standard triangle of $\ol{\cz}$ be given. Then in $\cc$, with $P\to Z$ as above,
we have a morphism of triangles, where the first and fourth vertical morphisms are identities
\[
\xymatrix{ \Omega Z \ar[d] \ar[r] & \Omega_r Z \ar[r]\ar[d] & P\ar[d] \ar[r] & Z \ar[d] \\
           \Omega Z \ar[r] & X \ar[r] & Y \ar[r] & Z.}
\]
Notice that
the second morphism is not canonical; in fact, any morphism
making the first square commutative lifts the given morphism in $\ol{\cz}$.
We will show that $(\ol{F},\phi)$ takes the triangle $\Omega_r Z \to X\to Y\to Z$
of $\ol{\cz}$ to a triangle of $\cc'$.
For this, we form a morphism of triangles in $\per(\Gamma)$
\[
\xymatrix{
\Omega q(Z) \ar[r] \ar[d] & O \ar[d] \ar[r] & qP \ar[d] \ar[r] & qZ \ar[d] \\
\Omega q(Z) \ar[r]        & X'\ar[r]        & qY \ar[r]        & qZ}.
\]
Its image under $\pi: \per(\Gamma) \to \cc$ becomes isomorphic to
the given morphism after possibly adding a morphism factoring through $\Omega Z \to P$
to the given morphism $\Omega Z \to X$. Thus, we may assume that
the image under $\pi$ is isomorphic to the given morphism. By the claim,
the image of this morphism under $L'q$ is then isomorphic to
\[
\xymatrix{
\Omega L'q(Z) \ar[r] \ar[d] & L'\tau_{\leq 0} O  \ar[d] \ar[r] & L'qP \ar[d] \ar[r] & L'qZ \ar[d] \\
\Omega L'q(Z) \ar[r]        & L'\tau_{\leq 0} X' \ar[r]        & L'qY \ar[r]        & L'qZ}.
\]
We deduce that $(\ol{F},\phi)$ takes the triangle $\Omega Z \to X\to Y\to Z$
to the triangle
\[
\Omega L'q(Z) \to L'\tau_{\leq 0} X' \to  L'qY \to L'qZ
\]
of $\cc'$.
\end{proof}

\subsection{Deleting a sink in global dimension $2$}
\label{ss:deleting-a-sink}
As a second example of localization, let us consider a finite-dimensional
basic algebra $A$ over an algebraically closed field $k$. Assume
that $P_i$ is the indecomposable projective module corresponding
to a sink $i$ of the quiver of $A$. Let $e_i$ be the corresponding
idempotent of $A$. Let $B=A/Ae_i A$. Then it is easy to check that
the projection map
\[
A \to B
\]
is a localization of dg categories. Indeed, the localizing subcategory
$\cn$ of $\cd(A)$ generated by $P_i$ consists of all coproducts
of shifted copies of $P_i$ and its right orthogonal subcategory $\cn^\perp$
is the localizing subcategory generated by the $P_j$, $j\neq i$.
Clearly, this subcategory is equivalent to $\cd(B)$ by the
functor $?\lten_A B$.

From now on, let us assume that $A$ (and thus $B$) are of global
dimension at most~$2$. Then $A$ is in particular homologically
smooth and by theorem~\ref{thm:localization-CY-completion},
we obtain a localization of the corresponding $3$-Calabi-Yau completions
\[
\Pi_3(A) \to \Pi_3(B).
\]
Using theorem~\ref{thm:3-CY-completion-of-2-dim-alg}, we can
identify these dg algebras with Ginzburg algebras $\Gamma_3(Q,W)$
and $\Gamma_3(Q',W')$. It is not hard to check that $Q'$ is obtained
from $Q$ by omitting the vertex corresponding to $i$ and all arrows
starting or ending at it and that $W'$ is obtained from $W$ by
deleting all cycles passing through this vertex. Thus, the
results of section~\ref{ss:deleting-a-vertex-is-localization} apply
and we obtain that if $\cc(Q',W')$ is $\Hom$-finite, then it
is the Calabi-Yau reduction \cite{IyamaYoshino08} of $\cc(Q,W)$
at the image of $e_i\Gamma_3(Q,W)$. This example was treated previously
by Amiot-Oppermann \cite{AmiotOppermann09}
using different methods.

\subsection{Generalized mutations} \label{ss:generalized-mutations}
Let $k$ be an algebraically closed
field and $Q$ a finite quiver (possibly with oriented cycles).
Let $W$ be a potential on $Q$. Let $T$ be a tilting module over $kQ$,
\ie a module such that if $B$ is the endomorphism algebra of $T$,
the derived functor
\[
?\lten_B T : \cd(B) \to \cd(kQ)
\]
is an equivalence, \cf \cite{Keller07a}. If $X$ is a projective resolution of $T$ as
a $B$-$kQ$-bimodule, then $?\ten_B X$ is a Morita functor from
the dg category of bounded complexes of finitely generated
projective $B$-modules to the corresponding category of $kQ$-modules.
This functor yields an isomorphism
\[
HC_0(B) \iso HC_0(kQ).
\]
We let $W_B\in HC_0(B)$ be the element corresponding to $W\in HC_0(kQ)$.
Let $c_B$ and $c$ be the images in Hochschild homology
of $W_B$ and $W$ under Connes' map $B$.
Then by theorem~\ref{thm:localization-defo-CY-completion},
we have an induced Morita functor
\[
\Pi_3(B,c_B) \to \Pi_3(kQ,c)
\]
and by theorem~\ref{thm:CY-completion-is-Ginzburg-algebra} and
theorem~\ref{thm:3-CY-completion-of-2-dim-alg},
we obtain an induced Morita functor between Ginzburg algebras
\[
\Gamma_3(Q',W'+W'') \to \Gamma_3(Q,W)\ko
\]
where the quiver $Q'$ is obtained from the quiver of $B$ by adding a
new arrow $\rho_r: j \to i$ for each minimal relation $r:i \to j$,
the potential $W'$ is
\[
W'=\sum \rho_r r
\]
and the potential $W''$ lifts $W_B$ along the surjection $kQ' \to B$
taking all arrows $\rho_r$ to zero.
This construction is linked to mutation of quivers with potentials in the
sense of \cite{DerksenWeymanZelevinsky08} as follows: Let $i$
be a vertex of $Q$ which is the source of at least one arrow
and let $T$ be the direct sum of the
projectives $P_j$, $j\neq i$, and of $T_i$ defined by
the exact sequence
\[
0 \to P_i \to \bigoplus_{\alpha: i\to j} P_j \to T_i\to 0 \ko
\]
where the sum is taken over all arrows $\alpha$ with source $i$ and the
corresponding component of the map from $P_i$ to the sum is the left
multiplication by $\alpha$. Then the passage from $B=\End(T)$ to $kQ$
is given by an APR-tilt \cite{AuslanderPlatzeckReiten79}.
In this case, one can check that $(Q',W')$ is the `pre-mutation' of $Q$ at $i$
in the sense of \cite{DerksenWeymanZelevinsky08}, \ie $Q'$
is obtained from $Q$ by
\begin{itemize}
\item[1)] adding an arrow $[\alpha \beta]: j \to l$ for each subquiver
\[
\xymatrix{l \ar[r]^\beta & i \ar[r]^\alpha & l}
\]
of $Q$ and
\item[2)] replacing each arrow $\beta: l \to i$ by an arrow $\beta^*: i \to l$
and each arrow $\alpha: i \to j$ by an arrow $\alpha^*: j \to i$;
\end{itemize}
and the potential $W'$ is equal to $[W] + \sum [\alpha\beta] \beta^* \alpha^*$
where $[W]$ is obtained from $W$ by replacing each occurrence of a composition
$\alpha\beta$ in a cycle passing through $i$ by $[\alpha\beta]$.

\appendix
\section{Ginzburg's algebra is Calabi-Yau of dimension three \\by Michel Van den Bergh}

\subsection{Introduction}
To a quasi-free algebra $A$ and an element $z\in A$ (a ``super
potential'') Ginzburg associates in \cite{Ginzburg06} a certain DG-algebra
 ${\mathfrak{D}(A,z)}$. He proves
that if  ${\mathfrak{D}(A,z)}$ has no negative cohomology then it is
$3$-Calabi-Yau (see \cite[Remark 5.3.2]{Ginzburg06} but beware that Ginzburg
uses homological grading). It was recently observed by Keller that
 ${\mathfrak{D}(A,z)}$ is
\emph{always}
$3$-Calabi-Yau.  Below we give a proof of this fact using the
formalism of non-commutative differential geometry.
\subsection{Notations and conventions}
Throughout we work over the semi-simple base ring $l=ke_1+\cdots+ke_d$
where $e_i^2=e_i$ and $k$ is a field. In other words all our rings $R$
are implicitly equipped with a ring homomorphism $l\r R$. Unadorned
tensor products are \emph{over $k$}.
\subsection{Pairings of bimodules}
Duality for bimodules is confusing so here we write out our
 conventions.
This is a copy of \cite[\S 3.1]{VandenBergh08b}. Let $B$ be an arbitrary graded
$k$-algebra. We equip $B\otimes B$ with the outer $B$-bimodule
structure.  If $Q$ is a graded $B$-bimodule then $Q^\ast$ is by
 definition
$\Hom_{B^e}(Q,B\otimes B)$. This is still a $B$-bimodule through the
 surviving
inner bimodule structure on $B\otimes B$.

A pairing (or bilinear map) between graded $B$-bimodules $P,Q$
is a homogeneous map of degree $n$
\begin{equation}
\label{ref-3.1-0}
\langle-,-\rangle:P\times Q\r B\otimes B
\end{equation}
such that $\langle p,-\rangle$ is linear for the outer bimodule
structure on $B\otimes B$ and $\langle -,q\rangle$ is linear for the
inner bimodule structure on $B\otimes B$. The obvious example is of
course when $P$ is the bimodule dual $Q^\ast$ of $Q$ and
 $\langle-,-\rangle$
is the evaluation pairing.  We say that the pairing is non-degenerate
if $P$, $Q$ are finitely generated graded projective bimodules and the
pairing induces an isomorphism $P\cong \Sigma^n (Q^\ast)$.
\begin{example}
\label{ref-3.1-1}
Let $P=\Sigma^n(B\otimes_l B)$, $Q=B\otimes_l B$. It is easy to see
 that the
pairing
\[
\langle a\otimes b,c\otimes
 d\rangle=(-1)^{|a||b|+|a||c|+|b||c|-n|c|}\sum_i ce_ib\otimes ae_id
\]
for $a,b,c,d\in B$ is well-defined and non-degenerate of degree $n$.
\end{example}
The opposite pairing
of $\langle-,-\rangle$ is defined by \def\opp{\operatorname{opp}}
\[
\langle-,-\rangle^{\opp}: Q\times P\r B\otimes B:(q,p)\mapsto
 (-1)^{(n+|p)(n+|q|)}\sigma \langle p,q\rangle
\]
where ``$\sigma$'' denotes the interchange operator: $\sigma(a\otimes
 b)=(-1)^{|a||b|}(b\otimes a)$.
So although the definition of a pairing of bimodules is asymmetric it
is not important which bimodule appears on the left or right.

If $P=Q$ then we say that a pairing $\langle-,-\rangle$ is
 (anti-)symmetric
if
\[
\langle p,p'\rangle=(-)\langle p,p'\rangle^{\opp}
\]
If $B$ is a DG-algebra and $P$, $Q$ are DG-bimodules then we say that
\eqref{ref-3.1-0} is a DG-pairing if it is compatible with the
 differential,
i.e.\ if
\[
d\langle p,q\rangle=\langle dp,q\rangle+(-1)^{|p|+n} \langle
 p,dq\rangle
\]
If a DG-pairing is non-degenerate then obviously it induces an
 isomorphism
of \emph{DG-modules} $P\cong \Sigma^n (Q^\ast)$.

\subsection{Differentials and double derivations}
\label{ref-4-2}
 If $B$ is a graded algebra then we denote by $\Omega_{B/l}$ the
bimodule of relative differentials for $B/l$. $\Omega_{B/l}$ fits in an
 exact sequence
\begin{equation}
\label{ref-4.1-3}
0\r\Omega_{B/l} \xrightarrow{\phi} B\otimes_l B\r B\r 0
\end{equation}
  We denote the generators of $\Omega_{B/l}$
by $Db$, $b\in B$ where $\phi(Db)=b\otimes 1-1\otimes b$.

With respect to signs \emph{we assume that $D$ has homological degree
 zero}.
If $B$ is equipped with a differential $d$ then we extend it to
 $\Omega_{B/l}$
by putting $d(Db)=D(db)$.

\medskip

Assume that $B$ is equipped with a graded double Poisson bracket of
 degree $n$
(see
\cite[\S 2.1]{VandenBergh08a}). Then  there is a well-defined anti-symmetric
 pairing of
degree on $\Omega_{B/l}$ of degree $n$
which is determined by
\[
\langle D\eta,D\xi\rangle=
\ldb \eta,\xi\rdb
\]
We define
$\mathbb{T}_{B/l}=\Omega^\ast_{B/l}$.
We may identify $\mathbb{T}_{B/l}$ with the bimodule of \emph{double
 derivations}
\[
\mathbb{T}_{B/l}=\Der_{B/l}(B,B\otimes B)
\]
If $b\in B$ and
$\delta\in \TT_{B/l}$ then we write $\delta(b)=\delta(b)'\otimes
\delta(b)''$.  $\TT_{B/l}$ contains a canonical element $E$ given by
\[
E(a)=\sum_i ae_i\otimes e_i-e_i\otimes e_i a
\]
\begin{remark} We may write $E(a)=[a,\xi]$ where $\xi=\sum_i e_i\otimes
 e_i
\in l\otimes l$. If, as in \cite{CrawleyBoeveyEtingofGinzburg07}, one works over a more
general separable $k$-algebra then one must replace $\xi$ by the
 separability
idempotent in $l^e$.
\end{remark}
\subsection{The graded cotangent bundle}
Now let $A$ be a quasi-free finitely generated $k$-algebra and put
$\TT A=T_A(\Sigma \TT_{A/l})$. According to \cite[\S 3.2]{VandenBergh08a} $\TT
 A$ carries a canonical
graded double Poisson bracket of degree $1$: the so-called double
 Schouten-Nijenhuis
bracket.\footnote{In \cite{VandenBergh08a} this bracket had degree $-1$ since
we used the opposite grading.}  Thus according to \S\ref{ref-4-2} we
 get an induced anti-symmetric
pairing on $\Omega_{\TT A/l}$ of degree $1$.
\begin{lemma}
\label{ref-5.1-4}
This pairing is non-degenerate.
\end{lemma}
\begin{proof}
This can be deduced from the fact that the double
Schouten-Nijenhuis bracket is actually induced from a bisymplectic form
 \cite{CrawleyBoeveyEtingofGinzburg07,VandenBergh08a}. To help the reader
let us give a proof here. We have a standard exact sequence
\[
0\r \TT A\otimes_A \Omega_{A/l}\otimes_A \TT A\xrightarrow{\alpha}
\Omega_{\TT A/l}\xrightarrow{\beta} \Sigma(\TT A\otimes_A
 \TT_{A/l}\otimes_A \TT A)\r 0
\]
with for $\omega\in \Omega_{A/l}$, $\delta\in \TT_{A/l}$
\begin{align*}
\alpha(1\otimes \omega\otimes 1)&=\omega\\
\beta(D\delta)&=1\otimes\delta\otimes 1
\end{align*}
Hence we have for $a\in A$
\[
\langle 1\otimes D\delta\otimes 1,\alpha(Da)\rangle=\langle
 D\delta,Da\rangle=
\ldb \delta,a\rdb=\delta(a)=\langle \beta(D\delta),Da\rangle
\]
where on the right we have the standard (non-degenerate) pairing
 between
$\TT_{A/l}$ and $\Omega_{A/l}$, extended to a (still non-degenerate)
 pairing between
$\TT A\otimes_A \TT_{A/l}\otimes_A \TT A$ and $\TT A\otimes_A
 \Omega_{A/l}\otimes_A \TT A$.
It follows that $\alpha$ and $\beta$ are adjoint.

Thus
one gets a commutative diagram
\[
\begin{CD}
0@>>> \TT A\otimes_A \Omega_{A/l}\otimes_A \TT A @>\alpha>>
\Omega_{\TT A/l} @>\beta>> \Sigma(\TT A\otimes_A \TT_{A/l}\otimes_A \TT
 A) @>>> 0\\
@. @| @VVV @| @.\\
0@>>> \TT A\otimes_A \TT_{A/l}^\ast\otimes_A \TT A @>>\beta^\ast>
\Sigma (\Omega_{\TT A/l}^\ast) @>>\alpha^\ast> \Sigma(\TT A\otimes_A
 \Omega_{A/l}^\ast\otimes_A \TT A) @>>> 0
\end{CD}
\]
Hence  the middle arrow is an isomorphism
\end{proof}
Now fix a ``super potential'' $z\in \sum e_i Ae_i$. Contraction with
$Dz$ defines a differential $d$ on $\TT A$ \cite{Ginzburg06} (see also
\cite[\S 3.1]{VandenBergh08b}). On generators we have
\begin{align*}
da&=0\qquad\text{for $a\in A$}\\
d\delta&=\delta(z)''\delta(z)'\qquad\text{for $\delta\in \TT_{A/l}$}
\end{align*}
We will denote resulting DG-algebra by $\TT(A,z)$.

\medskip

In the commutative case it is well-known that contraction with a
$1$-form is a derivation for the Gerstenhaber structure on the graded
cotangent bundle and hence in particular it is compatible with the
Schouten bracket. A similar result is true in the non-commutative
case.
\begin{lemma}
\label{ref-5.2-5} $\TT(A,z)$ is a DG-Gerstenhaber algebra with product
of degree zero and double bracket of degree one.
\end{lemma}
\begin{proof} We only need to check compatibility of the differential
with the double bracket. This can be done on generators. The only
non-trivial verification is
\begin{equation}
\label{ref-5.1-6}
d\ldb \delta,\Delta\rdb=\ldb d\delta,\Delta\rdb+\ldb \delta,d\Delta\rdb
\end{equation}
for $\delta,\Delta\in \TT_{A/l}$.

Following the notations of \cite[\S 3.2]{VandenBergh08a} we have
\[
\ldb \delta,\Delta\rdb=\ldb \delta,\Delta\rdb_l
+\ldb \delta,\Delta\rdb_r
\]
with
\begin{align*}
\ldb \delta,\Delta\rdb_l&=\ldb \delta,\Delta\rdb_l'\otimes \ldb
 \delta,\Delta\rdb_l''\in \TT_{A/l}\otimes A\\
\ldb \delta,\Delta\rdb_r&=\ldb \delta,\Delta\rdb_r'\otimes \ldb
 \delta,\Delta\rdb_r''\in A\otimes \TT_{A/l}
\end{align*}
so that we have
\begin{align*}
d\ldb \delta,\Delta\rdb&=d\ldb \delta,\Delta\rdb_l+d\ldb
 \delta,\Delta\rdb_r
\end{align*}
with
\begin{align*}
d\ldb \delta,\Delta\rdb_l&=\ldb \delta,\Delta\rdb_l'(z)''\ldb
 \delta,\Delta\rdb_l'(z)'\otimes
\ldb \delta,\Delta\rdb_l''\\
d\ldb \delta,\Delta\rdb_r&=\ldb \delta,\Delta\rdb_r'\otimes \ldb
 \delta,\Delta\rdb_r''(z)'' \ldb \delta,\Delta\rdb_r''(z)'
\end{align*}
By definition we have
\begin{align*}
  \ldb \delta,\Delta \rdb _l&=\sigma_{23}\circ((\delta\otimes
 1)\Delta-(1\otimes \Delta)\delta)\\
  \ldb \delta,\Delta\rdb {}_r&=\sigma_{12}\circ ((1\otimes
  \delta)\Delta-(\Delta\otimes 1)\delta)
\end{align*}
which after inspection becomes
\begin{align*}
d\ldb \delta,\Delta\rdb_l&=\Delta(z)''\delta(\Delta(z)')'\otimes
 \delta(\Delta(z)')''-\Delta(\delta(z)'')''\delta(z)'\otimes
 \Delta(\delta(z)'')'\\
d\ldb \delta,\Delta\rdb_r&=\delta(\Delta(z)'')'\otimes
 \delta(\Delta(z)'')''\Delta(z)'-
\Delta(\delta(z)')''\otimes \delta(z)''\Delta(\delta(z)')'
\end{align*}
On the other hand we have
\begin{align*}
\ldb d\delta,\Delta\rdb&=-\sigma \Delta(\delta(z)''\delta(z)')\\
&=-\Delta(\delta(z)'')''\delta(z)'\otimes \Delta(\delta(z)'')'
-\Delta(\delta(z)')''\otimes \delta(z)''\Delta(\delta(z)')'
\end{align*}
and
\begin{align*}
\ldb \delta,d\Delta\rdb&=\delta(\Delta(z)''\Delta(z)')\\
&=\delta(\Delta(z)'')'\otimes \delta(\Delta(z)'')''\Delta(z)'
+\Delta(z)''\delta(\Delta(z)')'\otimes \delta(\Delta(z)')''
\end{align*}
so that \eqref{ref-5.1-6} indeed holds.
\end{proof}
We immediately deduce
\begin{lemma} The pairing on $\Omega_{\TT A/l}$ is compatible with $d$.
\end{lemma}
\begin{proof} We have to prove for $\omega,\omega'\in\Omega_{{\TT
 A}/l}$
\[
d\langle \omega,\omega'\rangle=\langle d\omega,\omega'\rangle+
(-1)^{|\omega'|+1} \langle \omega,d\omega'\rangle
\]
One verifies that it is sufficient to check this on ${\TT A}$-bimodule
 generators of
$\Omega_{{\TT A}/l}$.
The only interesting case to consider is $\omega=D\delta$,
 $\omega'=D\Delta$
and $\delta,\Delta\in \TT_{A/l}$.  In that case
the result is a direct concequence of Lemma \ref{ref-5.2-5} and in
particular \eqref{ref-5.1-6}.
\end{proof}

\subsection{Ginzburg's algebra}
Let $A,z,{{\TT } A}$ be as in the previous section. We have $E\in {\TT
 }_{A/l}\subset {{\TT } A}$.
We immediately check that $dE=0$. So $E$ defines a (presumably always
non-trivial)
cohomology class in ${{\TT } A}$. Ginzburg's idea is to kill this class
 through
adjunction of an extra variable $c$ of degree $-2$ commuting with $l$.
 So Ginzburg's
algebra is
\[
{\mathfrak{D}(A,z)}={{\TT } (A,z)}\ast_l l[c]
\]
where $|c|=-2$  and
$
dc=E
$.
To simplify the notations we will write ${\TT }={\TT }(A,z)$ and
 $\mathfrak{D}=
\mathfrak{D}(A,z)$ in this section.

We have a presentation
\[
0\r \Omega_{{\mathfrak{D}}/l}\xrightarrow{\phi} {\mathfrak{D}}\otimes_l
 {\mathfrak{D}}\r {\mathfrak{D}}\r 0
\]
where $\phi$ is as in \eqref{ref-4.1-3}.
It is easy to see that as graded ${\mathfrak{D}}$-bimodule we have
\[
\Omega_{{\mathfrak{D}}/l}=({\mathfrak{D}}\otimes_{\TT } \Omega_{{\TT
 }/l}\otimes_{\TT } {\mathfrak{D}})\oplus
 ({\mathfrak{D}}\otimes_l l Dc\otimes_l {\mathfrak{D}})
\]
\def\II{\mathbb{I}}\def\cone{\operatorname{cone}}
Put $\II=\sum_i e_i\otimes e_i$. Then ${\mathfrak{D}}$ is
 quasi-isomorphic to $\cone \phi$ and
$\cone \phi$ is given by
\[
P=({\mathfrak{D}}\otimes_l l\II \otimes_l {\mathfrak{D}} )
\oplus \Sigma({\mathfrak{D}}\otimes_{\TT } \Omega_{{\TT
 }/l}\otimes_{\TT } {\mathfrak{D}})
\oplus \Sigma ({\mathfrak{D}}\otimes_l l Dc\otimes_l {\mathfrak{D}})
\]
with total differential
\begin{align*}
d_P\mathbb{I}&=0\\
d_P\omega&=\phi_{\TT }(\omega)-d_{\TT }\omega\qquad \text{for
 $\omega\in \Omega_{\TT }$}\\
d_P(Dc)&=[c,\II]- D(E)
\end{align*}
We define a \emph{symmetric} pairing of degree $3$ on
$P$ by putting
\begin{align*}
\langle Dc,\mathbb{I}\rangle_P&=\sum_i e_i\otimes e_i\\
\langle \omega,\omega'\rangle_P&=(-1)^{|\omega|_{\TT }-1}\langle
 \omega,\omega'\rangle_{\TT }
\end{align*}
and assigning the value zero on other combinations of  generators of
 $P$
taken from $\II,\Omega_{{\TT }/l},Dc$.  Note
that in $P$ we have $|\II|=0$, $|Dc|=-3$ and $|\omega|_P=|\omega|_{\TT
 }-1$
for $\omega\in \Omega_{{\TT }/l}$. The requirement of symmetry yields
\begin{align*}
\langle \II,Dc\rangle_P&=(-1)^{(|\II|+3)(|Dc|+3)}\sigma \langle
 Dc,\II\rangle_P\\
&=\sum_i e_i\otimes e_i
\end{align*}
By combining Example \ref{ref-3.1-1}
with Lemma \ref{ref-5.1-4} we see that $\langle-,-\rangle_P$ is
 non-degenerate.

We claim that $\langle-,-\rangle_P$ is compatible with the
 differential. By
symmetry this amounts to six verifications which we now carry out.

\medskip

\noindent \textbf{Case 1 } One has
\[
d_{\mathfrak{D}}\langle Dc,Dc\rangle_P=0
\]
and
\begin{align*}
\langle d_P Dc,Dc\rangle_P&=\langle [c,\II]- D(E),Dc\rangle_P\\
&=\sum_i(e_i\otimes ce_i-e_ic\otimes e_i)\\
&=\sum_i(e_i\otimes e_ic-ce_i\otimes e_i)
\end{align*}
and
\begin{align*}
\langle Dc,d_P Dc\rangle_P&=\langle Dc,[c,\II]- D(E)\rangle_P\\
&=\sum_i(ce_i\otimes e_i-e_i\otimes e_ic)
\end{align*}
so that
\[
d_{\mathfrak{D}}\langle Dc,Dc\rangle_P=\langle d_P
 Dc,Dc\rangle_P+(-1)^{|Dc|+3}\langle Dc,d_P Dc\rangle_P
\]
\noindent \textbf{Case 2 } One has for $u\in {\TT }$
\[
d_{\mathfrak{D}}\langle Dc,Du\rangle_P=0
\]
and
\begin{align*}
\langle d_P Dc,Du\rangle_P&=\langle [c,\II]- D(E),Du\rangle_P\\
&=-(-1)^{|E|_{\TT }-1}\langle D(E),Du\rangle_{\TT }\\
&=-\ldb E, u\rdb\\
&=-\sum_i(ue_i\otimes e_i-e_i\otimes e_i u)
\end{align*}
and
\begin{align*}
\langle Dc,d_PDu\rangle_P &=\langle Dc,[u,\II]-Dd_{\TT }u\rangle\\
&=\sum_i ue_i\otimes e_i-e_i\otimes e_i u
\end{align*}
so that
\[
d_{\mathfrak{D}}\langle Dc,Du\rangle_P=\langle d_P
 Dc,Du\rangle_P+(-1)^{|Dc|+3}\langle Dc,d_P Du\rangle_P
\]
\noindent \textbf{Case 3 } One has
\[
d_{\mathfrak{D}}\langle Dc,\II\rangle_P=0
\]
and
\begin{align*}
\langle d_P Dc,\II\rangle_P&=\langle [c,\II]-D(E),\II\rangle_P\\
&=0
\end{align*}
and
\begin{align*}
\langle Dc,d_P \II\rangle_P&=0
\end{align*}
Hence this case is trivial.

\medskip

\noindent \textbf{Case 4 } One has for $\omega,\omega'\in \Omega_{{\TT
 }/l}$
\begin{align*}
d_{\mathfrak{D}}\langle \omega,\omega'\rangle_P&=(-1)^{|\omega|_{\TT
 }-1}d_{\TT }\langle \omega,\omega'\rangle_{\TT }\\
&=(-1)^{|\omega|_{\TT }-1}\langle d_{\TT }\omega,\omega' \rangle_{\TT
 }+(-1)^{|\omega|_{\TT }-1}
(-1)^{|\omega|_{\TT }+1}\langle \omega,d_{\TT }\omega' \rangle_{\TT }\\
&=(-1)^{|\omega|_{\TT }-1}\langle d_{\TT }\omega,\omega' \rangle_{\TT
 }+\langle \omega,d_{\TT }\omega' \rangle_{\TT }
\end{align*}
and
\begin{align*}
\langle d_P\omega,\omega'\rangle_P&=
\langle \phi_{\TT }(\omega)-d_{\TT }\omega,\omega'\rangle_P\\
&=-(-1)^{|\omega_{\TT }|+1-1}\langle\langle d_{\TT
 }\omega,\omega'\rangle_{\TT }\\
&=(-1)^{|\omega_{\TT }|-1}\langle d_{\TT }\omega,\omega'\rangle_{\TT }
\end{align*}
and
\begin{align*}
\langle \omega,d_P\omega'\rangle_P&=
\langle\omega, \phi_{\TT }(\omega')-d_{\TT }(\omega') \rangle_P\\
&=-(-1)^{|\omega|_{\TT }-1}\langle\omega, d_{\TT }(\omega')\rangle_{\TT
 }\\
&=(-1)^{|\omega|_{\TT }}\langle  \omega,d_{\TT }(\omega) \rangle_{\TT }
\end{align*}
so that we get
\[
d_{\mathfrak{D}}\langle \omega,\omega'\rangle_P=\langle
 d_P\omega,\omega'\rangle_P+
(-1)^{|\omega|_{\TT }}\langle \omega,d_P\omega'\rangle_P
\]
which is correct since $|\omega|_{\TT }=|\omega|_P+3(\mod 2)$.

\medskip

\noindent \textbf{Case 5 } One has for $\omega\in \Omega_{{\TT }/l}$
\begin{align*}
d_{\mathfrak{D}}\langle \omega,\II\rangle_P=0
\end{align*}
and
\begin{align*}
\langle d_P\omega,\II\rangle_P&=\langle \phi_{\TT
 }(\omega)-d_c(\omega),\II\rangle\\
&=0
\end{align*}
and
\begin{align*}
\langle \omega,d_P\II\rangle_P&=0
\end{align*}
So nothing to prove here!

\medskip

\noindent \textbf{Case 6 } The last case is about
 $\langle\II,\II\rangle_P$
but this is trivial.

\medskip

We can now conclude
\begin{theorem} The Ginzburg algebra ${\mathfrak{D}}$ is
 $3$-Calabi-Yau.
\end{theorem}
\begin{proof} We need to prove
\begin{equation}
\label{ref-6.1-7}
\RHom_{{\mathfrak{D}}^e}({\mathfrak{D}},{\mathfrak{D}}\otimes
 {\mathfrak{D}})\cong \Sigma^{-3} {\mathfrak{D}}
\end{equation}
in $D({\mathfrak{D}}^e)$ and moreover this isomorphism must be self
 dual. We have
\begin{align*}
\RHom_{{\mathfrak{D}}^e}({\mathfrak{D}},{\mathfrak{D}}\otimes
 {\mathfrak{D}})&\cong\Hom_{{\mathfrak{D}}^e}(P,{\mathfrak{D}}\otimes
 {\mathfrak{D}})\\
&\cong \Sigma^{-3} P\\
&\cong \Sigma^{-3} {\mathfrak{D}}
\end{align*}
where the second isomorphism is obtained from the pairing
$\langle-,-\rangle_P$. Self duality follows from the fact
that $\langle-,-\rangle_P$ is symmetric.
\end{proof}
\subsection{A word on quivers}
Assume now that $V$ is a finitely generated $l$-bimodule and put
$A=T_l V$. Thus $A$ is the path algebra of a quiver. We remind
the reader on the concrete interpretation of ${\mathfrak{D}(A,z)}$ in
 this case.
This is taken from \cite{Ginzburg06}. Let $(t_i)_i$ a $k$-basis of $V$
where for each $i$ we have $t(i)$, $h(i)$ such that $t_i\in
 e_{t(i)}Ve_{h(i)}$.

Then we may define operations
\begin{gather*}
\left(\frac{\partial\ }{\partial t^i}\right)^+:A/[A,A]\r A\\
\frac{\partial\ }{\partial t^i}:A\r A\otimes A
\end{gather*}
where the second one is the element of ${\TT }_{A/l}$ with the property
\[
\frac{\partial t^j}{\partial t^i}=
\delta^{ij}(e_{t(i)}\otimes e_{h(i)})
\]
and the first one is obtained from the first by the following
 commutative
diagram
\[
\begin{CD}
A @>>> A/[A,A]\\
@V\frac{\partial\ }{\partial t^i} VV @VV\left(\frac{\partial\
 }{\partial t^i}\right)^+   V\\
A \otimes A @>a\otimes b \mapsto ba>> A
\end{CD}
\]
By \cite[Prop.\ 6.2.2(2)]{VandenBergh08a} we have
\begin{equation}
\label{ref-7.1-8}
E=\sum_i \left[\frac{\partial\ }{\partial t^i},t^i\right]
\end{equation}
as elements of ${\TT }_{A/l}$.

\medskip

Pick $z\in \oplus_i e_iAe_i$.
\begin{lemma}\cite{Ginzburg06} As graded algebras there is a canonical
 isomorphism
\[
{\mathfrak{D}(A,z)}=T_l(V\oplus \Sigma V^\ast\oplus kc)
\]
Furthermore if $t^i$ is the dual basis to $t_i$ then the differential
on ${\mathfrak{D}(A,z)}$ is given by
\begin{equation}
\label{ref-7.2-9}
\begin{aligned}
dt^i&=0\\
dt_i&=\left(\frac{\partial z}{\partial x^i}\right)^+\\
dc&=\sum_i \left[t_i,t^i\right]
\end{aligned}
\end{equation}
\end{lemma}
\begin{proof}
Put $t_i=\frac{\partial\ }{\partial t^i}$.
We get ${\TT(A,z) }=T_l(V\oplus V^\ast)$ where $(t^i)$ is the basis for
 $V^\ast$,
dual to $(t_i)_i$.

The differential
$d$ on ${\TT(A,z) }$ has the property.
\begin{align*}
dt^i&=0\\
dt_i&=\left(\frac{\partial z}{\partial x^i}\right)^+
\end{align*}
Finally the algebra ${\mathfrak{D}(A,z)}$ is obtained by adjoining $c$
 such that
\[
dc=E=\sum_i \left[t_i,t^i\right]\qed
\]
where we have used \eqref{ref-7.1-8}. \end{proof}
\subsection{A word on $\Ext$-algebras}
The advantage of the presentation \eqref{ref-7.2-9} is that
we can immediately read off the $A_\infty$-structure on the
$\Ext$-algebra of ${\mathfrak{D}(A,z)}$. This works more generally as
 follows. Assume
that $W$ is a finite dimensional $l$-bimodule and we have
a DG-algebra structure on $B=T_lW$ compatible with the canonical
 augmentation
$B\r l$. Then for $w\in W$ we may write
\[
dw=\sum_{n=1}^\infty b_n^\ast(w)
\]
where  the $b_n^\ast$ are maps
\[
 b_n^\ast:W\r W^{\otimes n}
\]
of degree $1$. Dualizing we get maps of degree $1$
\[
b_n:(W^\ast)^{\otimes n}\r W^\ast
\]
which define an $A_\infty$ structure on $\Sigma^{-1}(W^\ast)$ (without
unit). It follows
from tbe bar-cobar machinery  that the $A_\infty$-algebra
$l\oplus \Sigma^{-1}(W^\ast)$ corresponds to $\RHom_B(l,l)$.

Now let $V,A,z,{\mathfrak{D}(A,z)}$ be as before
and assume that $z$ contains no linear terms. We put $W=V\oplus \Sigma
 V^\ast
\oplus kc$. Thus ${\mathfrak{D}(A,z)}=T_lW$ and the $\Ext$-algebra of
 ${\mathfrak{D}(A,z)}$ as a graded
vector space\footnote{If $z$ contains quadratic terms then
this algebra has a non-trivial differential so it is not strictly
 speaking
the $\Ext$-algebra} is $l\oplus \Sigma^{-1}W^\ast=l\oplus \Sigma^{-1}
 V^\ast\oplus \Sigma^{-2} V \oplus k\Sigma^{-1} (c^\ast)$.

One checks that the $A_\infty$-operations are the pairings
 $V^\ast\otimes V\r l$ and
$V\otimes V^\ast\r l$ as well $n$-ary operations $(V^\ast)^{\otimes n}
 \r \Sigma^{-1} V$ which are obtained from  the degree $n+1$-part
 $z_{n+1}\in V^{\otimes(n+1)}$
of the superpotential $z$.


\def\cprime{$'$} \def\cprime{$'$}
\providecommand{\bysame}{\leavevmode\hbox to3em{\hrulefill}\thinspace}
\providecommand{\MR}{\relax\ifhmode\unskip\space\fi MR }
\providecommand{\MRhref}[2]{%
  \href{http://www.ams.org/mathscinet-getitem?mr=#1}{#2}
}
\providecommand{\href}[2]{#2}

\end{document}